\documentclass{cimart}

\usepackage{graphicx,color}
\usepackage[dvipsnames]{xcolor}
\usepackage{float}
\usepackage{pgf,tikz,pgfplots}
\usepgfplotslibrary{polar}
\usetikzlibrary{automata,positioning,arrows.meta}
\usetikzlibrary{arrows}

\usepackage{makeidx}

\newtheorem{assumption}[definition]{Assumption}

\newcommand{\C}{\mathbb C}
\newcommand{\N}{\mathbb N}
\newcommand{\R}{\mathbb R}
\DeclareMathOperator{\codim}{codim}
\DeclareMathOperator{\dist}{dist}
\DeclareMathOperator{\divergence}{div}
\DeclareMathOperator{\spa}{span}
\DeclareMathOperator{\subspace}{subspace}
\DeclareMathOperator{\supp}{supp}
\DeclareMathOperator*{\esssup}{ess\,sup}
\DeclareMathOperator{\trace}{tr}

\title{
    Geometric spectral theory of quantum graphs
    }

\author{
    James Kennedy
    }

\authorinfo[
    J. Kennedy]{
    Department of Mathematics, Faculty of Sciences of the University of Lisbon, Lisboa, Portugal.  
    }{%
    jbkennedy@ciencias.ulisboa.pt
    }

\abstract{%
These are lecture notes from a course given at the summer school ``Heat kernels and spectral geometry: from manifolds to graphs'' in Bregenz, Austria, 2022. They are designed to be accessible to doctoral level students, and include background chapters on Laplacians on domains and quantum graphs before moving on to specialised topics involving the dependence and optimisation of operator eigenvalues on a metric graph in function of the graph geometry, drawn in part from the recent literature. 
    }



\VOLUME{32}
\ISSUE{3}
\YEAR{2024}
\firstpage{393}
\DOI{https://doi.org/10.46298/cm.12380}

\begin{document}

\section{Introduction} 

These lecture notes are a moderately edited version of course notes written for an eponymous short course I gave at a summer school entitled ``Heat kernels and spectral geometry: from manifolds to graphs'', which drew around 50 mathematicians, mostly doctoral students and postdocs, to Bregenz, Austria, for a week in August/September 2022. The original version of these notes, distributed to all summer school participants, covered exactly the same material as the lectures, but provided more detail and background in a number of places (as well as references and exercises); it may be accessed online at \cite{kennedylec}. An evolved version of the original notes, quite similar to the current version but without the exercises, was also submitted in the scope of my Portuguese \emph{agrega\c{c}\~ao} (habilitation) exams.\medskip

These notes centre around how the spectrum of a quantum graph (i.e.~a Schrödinger-type differential operator defined on a metric graph) depends on the geometric and topological properties of the graph. However, the notes are aimed at non-specialists with no prior exposure to quantum graphs; 
in principle no more than a solid grounding in analysis and functional analysis is needed, although some exposure to the modern theory of (linear) partial differential equations is highly advantageous.

As such, the first two chapters are introductory in nature, covering the basics of weak solution theory and the spectral theorem for the Laplacian (as one might see in an advanced master's course) and a construction of metric graphs and differential operators on them, respectively. As such, few to no citations have been included on the understanding that this material is standard/well known/folklore, although no claim is made to originality or authorship of any part (the particular organisation and presentation of the topics is my own, but unlikely original). Much of Chapter~\ref{chapter:domains} is likely similar in content to parts of \cite{brezis}, which was suggested background reading for the participants, although \cite{arendturban} is probably closer to my own way of thinking. Likewise, Chapter~\ref{chapter:quantum-graphs} was (with a few, marked, exceptions) not written following any particular source, but was likely influenced by a number of works (including ones partly due to me). Possible influences include, but are by no means limited to, \cite{band17,berkolaiko16,bkkm1,bkkm2,berkolaiko,kklm,mugnolo}.\medskip

The core of the matter is Chapter~\ref{ch:geometric-spectral-theory}, which covers the evolution of eigenvalue estimates for quantum graph Laplacians over the last 35 years (including a brief apology for the study of geometric spectral theory in Section~\ref{sec:gst-domains}, as well as references to related topics such as Weyl asymptotics or the study of nodal domains and nodal counts). This draws far more heavily on recent literature and contains no new research, but I believe the synthesis as such is original. For the current state of the art, see, for example, \cite{bkkm2,bkkm3}.\medskip

An attempt has been made to give a balanced overview of the history and main theorems in the quite specific area of estimates for Laplacian eigenvalues based on the geometry of the graph; however, the choice of topics, in particular the focus on eigenvalue bounds and ``surgery''-type methods for examining the relation between the spectrum and the graph, is somewhat idiosyncratic and reflects my own interests, and is arguably weighted towards my own contributions precisely because of that.\medskip

There are a few deliberate changes from the original version to improve readability and usefulness as a standalone reference, roughly in keeping with the changes introduced for the \emph{agrega\c{c}\~ao} version; in particular, a few previously ``staccato'' passages of text and motivations have been filled out, and a bit more historical context is given in several places. More significantly, I have also introduced a few more recent and advanced results in whose creation I was involved, which will inevitably be in more of a ``survey article'' style than a ``lecture notes'' style. To quarantine any effect on the balance and flow of the lecture notes these are collected in a new final Section~\ref{sec:adv-surgery} not present in the original version. On the other hand, the exercises originally included for the benefit of the participants have been kept, and are in an appendix to this version.\medskip

The summer school was organised by Delio Mugnolo (Hagen) and Pavel Kurasov (Stockholm) within the framework of Action CA18232 \emph{Mathematical models for interacting dynamics of networks} of the COST Association (European Cooperation in Science and Technology), see \url{https://mat-dyn-net.eu/}, focused, as the title suggests, on various aspects of difference and differential equations on graphs and networks. Each of the three courses (including, of course, the one which gave rise to these lecture notes) consisted of four 90-minute lectures given in the mornings; there was free time in the afternoons to work on the exercises.\medskip

My thanks go, naturally, to the two organisers who are also collaborators on and around the topic of these notes, their organisational assistant Anna Liza Schonlau, the summer school participants themselves for their constructive feedback, comments and suggestions, as well as many other past and present co-authors and colleagues, including but not limited to Gregory Berkolaiko, Marco Düfel, Matthias Hofmann, Corentin Lena, Gabriela Malenov\'a, Marvin Plümer, Jonathan Rohleder, Andrea Serio and Matthias Täufer.\medskip

Funding for the summer school was provided by the European Union via the COST Association, as well as the University of Hagen, Germany. My own work in this area over the last six or so years was also supported by the Portuguese Science Foundation (Funda\c{c}\~ao para a Ci\^encia e a Tecnologia) via the grants IF/01461/2015, PTDC/MAT-CAL/4334/2014, PTDC/MAT-PUR/1788/2020 and UIDB/00208/2020, and before that by the Alexander von Humboldt Foundation, Germany.\medskip



%
%

%
%
\section{Laplacians on domains}
\label{chapter:domains}

\subsection{From forms to operators}
\label{sec:forms}

We start by recalling the form approach to unbounded linear operators and the spectral theorem.

\begin{assumption}
$V$ and $H$ are separable Hilbert spaces such that $V$ is compactly and densely embedded in $H$. (That is, there is a compact injective mapping $i: V \to H$ such that $i(V)$ is dense in $H$.)
\end{assumption}

We let $a: V \times V \to \C$ be a Hermitian form (symmetric sesquilinear form), that is,
\begin{displaymath}
\begin{aligned}
	a(\alpha u + \beta v,w) &= \alpha a(u,w) + \beta a(v,w) \qquad &&\text{for all } u,v,w \in V,\, \alpha,\beta \in \C,\\
	a(u,v) &= \overline{a(v,u)} \qquad &&\text{for all } u,v \in V,
\end{aligned}
\end{displaymath}
(in particular, $a(u,u) \in \R$ for all $u \in V$), and say that $a$ is:
\begin{itemize}
\item \emph{bounded} if there exists $M>0$ such that
\begin{displaymath}
	|a(u,v)| \leq M\|u\|_V \|v\|_V \qquad \text{for all } u,v\in V;
\end{displaymath}
\item \emph{$H$-elliptic} if there exist $\xi \geq 0$ and $\omega > 0$ such that
\begin{displaymath}
	a(u,u) + \xi \|u\|_H^2 \geq \omega \|u\|_V^2 \qquad \text{for all } u\in V;
\end{displaymath}
if we can take $\xi=0$, then we say $a$ is \emph{coercive}.
\end{itemize}
If $a$ has these properties, then $\sqrt{a(u,u) + \xi \|u\|_H^2}$ defines an equivalent norm on $V$.

\begin{example}
If $V=H=\C^d$ and $A \in \C^{d\times d}$ is any Hermitian matrix (i.e. $A^T=\overline{A}$), then
\begin{displaymath}
	a(u,v) := u^T Av, \qquad u,v \in \C^d,
\end{displaymath}
defines a (trivially bounded) Hermitian form on $\C^d \times \C^d$; in fact, there is a one-to-one correspondence between (Hermitian) forms $a$ and matrices $A$.
\end{example}

\begin{definition}
Let $a: V \times V \to \C$ be a bounded Hermitian form. The operator $A: D(A) \subset H \to H$ associated with $a$ is defined by:
\begin{displaymath}
\begin{aligned}
	D(A) &:= \{u \in V: \text{there exists } f \in H \text{ such that } a(u,v) = \langle f,v\rangle_H \text{ for all } v \in V\},\\
	Au &:= f.
\end{aligned}
\end{displaymath}
\end{definition}
That is, $A$ is given by the rule $\langle Au,v\rangle_H = a(u,v)$ for all $v \in V$, where $u \in D(A)$ iff there exists a vector $Au \in H$ with this property.

\begin{theorem}[Spectral theorem]
\label{thm:spectral}
Let $a$ be a bounded, $H$-elliptic Hermitian form. Then $A$ is a self-adjoint operator, $D(A)$ is dense in $H$, and $A$ has compact resolvent. (That is, for any $\lambda \in \C$ such that $A-\lambda I$ is invertible, $(A-\lambda I)^{-1}$ is compact, and its resolvent set, the set of all such $\lambda \in \C$, is nonempty.) In particular, its spectrum $\sigma (A) \subset \C$ is a discrete set of real eigenvalues of finite (algebraic $=$ geometric) multiplicity of the form
\begin{displaymath}
	-\xi \leq \lambda_1 \leq \lambda_2 \leq \lambda_3 \leq \ldots \to \infty.
\end{displaymath}
The corresponding eigenvectors $u_k \sim \lambda_k$ can be chosen real, and to form an orthonormal basis of $H$; the pair $(\lambda_k,u_k)$ satisfies both the strong eigenvalue equation
\begin{displaymath}
	Au_k = \lambda_k u_k
\end{displaymath}
and the weak form of the equation
\begin{displaymath}
	a(u_k,v) = \lambda_k \langle u_k,v\rangle_H \qquad \text{for all } v \in V.
\end{displaymath}
\end{theorem}

In particular, $A$ can be diagonalised: if
\begin{displaymath}
	v = \sum_{k=1}^\infty \langle v,u_k\rangle_H u_k,
\end{displaymath}
then $v \in D(A)$ iff $\sum_{k=1}^\infty \lambda_k^2 \langle v,u_k\rangle_H^2 < \infty$, and in this case
\begin{displaymath}
	Av = \sum_{k=1}^\infty \lambda_k \langle v,u_k\rangle_H u_k.
\end{displaymath}

\begin{theorem}[Courant--Fischer/min-max and max-min characterisation of the eigenvalues]
\label{thm:courant-fischer}
Under the assumptions of Theorem~\ref{thm:spectral},
\begin{subequations}
\begin{align}
	\lambda_1 &= \min_{0 \neq u \in V} \frac{a(u,u)}{\|u\|_H^2} = \min_{0 \neq u \in D(A)} \frac{\langle Au,u\rangle_H}{\|u\|_H^2}\label{eq:l1} \\
	\lambda_k &=\min_{u \perp u_1,\ldots,u_{k-1}} \frac{a(u,u)}{\|u\|_H^2}
	= \min_{\substack{U \subset V \subspace \\ \dim U = k}}\,\,\max_{0 \neq u \in U} \frac{a(u,u)}{\|u\|_H^2}\label{eq:lk-minmax}\\
	&=\max_{\substack{W \subset V \subspace \\ \codim W = k-1}}\,\, \max_{0 \neq u \in W^\perp} \frac{a(u,u)}{\|u\|_H^2}.\label{eq:lk-maxmin}
\end{align}
\label{eq:courant-fischer}
\end{subequations}
Equality is attained in equation \eqref{eq:lk-minmax} if $U=\spa\{u_1,\ldots,u_k\}$ and in equation \eqref{eq:lk-maxmin} if $W = (\spa \{u_1,\ldots,u_{k-1}\})^\perp$.
\end{theorem}

We call the quotient $\frac{a(u,u)}{\|u\|_H^2}$ the \emph{Rayleigh quotient} of $u \in V$.

\subsection{The spectral theorem for the Laplacian}
\label{sec:laplacian}

Here we suppose $\Omega \subset \R^d$, $d\geq 1$ is a bounded domain (connected open set) with sufficiently smooth boundary $\partial\Omega$, and write $x = (x_1,\ldots,x_d) \in \R^d$, as well as $\nu(x) \in \mathbb{S}^{d-1}$ for the outward-pointing unit normal to $\Omega$ at $x \in \partial\Omega$. (Throughout most of this section $\Omega$ may be replaced by a more general smooth Riemannian manifold with no essential changes to the main results.) For a (twice differentiable) function $u: \Omega \to \C$, we write
\begin{equation}
\label{eq:laplacian}
	\Delta u(x) := \sum_{j=1}^d \frac{\partial^2 u}{\partial x_j^2} (x) = \trace D^2u(x)
\end{equation}
for the \emph{Laplacian} of $u$ at $x \in \Omega$. We also need the following spaces of functions (here and throughout all integrals are understood as Lebesgue integrals, and the measure is always taken to be $d$-dimensional Lebesgue measure):
\begin{displaymath}
	C_c^\infty (\Omega) = \{u \in C^\infty (\overline{\Omega}): \supp u \Subset \Omega\}
\end{displaymath}
is the set of smooth functions whose support $\supp u = \overline{\{x \in \Omega: u(x) \neq 0\}}$ is a compact subset of $\Omega$. Equipped with the norm
\begin{displaymath}
	\|u\|_{L^p (\Omega)} \equiv \|u\|_p := \begin{cases} \left(\int_\Omega |u|^p\,\textrm{d}x\right)^{1/p}, \qquad & 1\leq p <\infty,\\
	\esssup_{x\in\Omega} |u(x)|,\qquad & p=\infty,\end{cases}
\end{displaymath}
the space
\begin{displaymath}
	L^p(\Omega) = \{u : \Omega \to \C \text{ measurable}: \|u\|_p < \infty\}
\end{displaymath}
is a Banach space for all $1\leq p\leq \infty$, and a Hilbert space for $p=2$ with respect to the corresponding inner product
\begin{displaymath}
	\langle f,g\rangle = \int_\Omega f\overline{g}\,\textrm{d}x,
\end{displaymath}
which will play the role of the space $H$ from Section~\ref{sec:forms}.

To define our $V$, we need to introduce Sobolev spaces, i.e. spaces of weakly differentiable functions. (At an intuitive level it is often sufficient to imagine these functions as being piecewise-smooth functions, which may have ``kinks'', and to know that (1) all basic theorems from vector calculus continue to hold for them, and (2) that they form a complete space with respect to the natural norm of the space. In particular, we can apply the spectral theorem to suitable operators defined on them.)

\begin{definition}
\label{def:weak-derivative}
The locally integrable function $g_j : \Omega \to \C$ is a \emph{weak} (or \emph{distributional}) partial derivative of the locally integrable function $u$ if it satisfies the integration by parts formula
\begin{equation}
\label{eq:weak-derivative-1}
	\int_\Omega g_j \phi\,\textrm{d}x = -\int_\Omega u \frac{\partial \phi}{\partial x_j}\,\textrm{d}x
\end{equation}
for all $\phi \in C_c^\infty (\Omega)$. In this case we write $\frac{\partial u}{\partial x_j} := g_j$, and
\begin{displaymath}
	\nabla u = \left(\frac{\partial u}{\partial x_1},\ldots,\frac{\partial u}{\partial x_d}\right)^T
\end{displaymath}
if all weak partial derivatives $\frac{\partial u}{\partial x_j}$ exist, $j=1,\ldots,d$.
\end{definition}
One can show that weak derivatives, if they exist, are unique up to a set of measure zero.

Higher order weak partial derivatives are defined accordingly: for any multi-index $\alpha =~(\alpha_1,\ldots,\alpha_d) \in \N_0^d$, of order $|\alpha| = \alpha_1 + \ldots + \alpha_d$, $\frac{\partial^\alpha u}{\partial x^\alpha} = \frac{\partial^{|\alpha|}u}{\partial x_1^{\alpha_1}\ldots \partial x_d^{\alpha_d}}$ is defined to be the locally integrable function $g_\alpha$, if one exists, such that
\begin{displaymath}
	\int_\Omega g_\alpha \phi\,\textrm{d}x = (-1)^{|\alpha|}\int_\Omega u \frac{\partial^\alpha \phi}{\partial x^\alpha}\,\textrm{d}x
\end{displaymath}
for all $\phi \in C_c^\infty (\Omega)$. In particular, we may define the Laplacian \eqref{eq:laplacian} of a function $u$ in the weak sense, as a sum of weak second-order derivatives.

\begin{definition}[Sobolev spaces]
For $k\geq 0$ and $1\leq p\leq \infty$, we define
\begin{displaymath}
	W^{k,p}(\Omega) = \{u \in L^p(\Omega): \text{all weak derivatives of $u$ up to order $k$ exist and are in $L^p(\Omega)$}\}
\end{displaymath}
as well as the corresponding norm
\begin{displaymath}
	\|u\|_{W^{k,p}(\Omega)} := \left(\|u\|_p^p + \sum_{1\leq |\alpha|\leq k} \left\|\frac{\partial^\alpha u}{\partial x^\alpha}\right\|_p^p\right)^{1/p}
	= \left(\int_\Omega |u|^p\,\textrm{d}x + \sum_{1\leq|\alpha|\leq k} \int_\Omega \left|\frac{\partial^\alpha u}{\partial x^\alpha}\right|^p\textrm{d}x
	\right)^{1/p}.
\end{displaymath}
Finally, we set
\begin{displaymath}
	W^{k,p}_0 (\Omega) := \overline{C_c^\infty (\Omega)}^{\|\cdot\|_{W^{k,p}(\Omega)}}
\end{displaymath}
to be the subspace of $W^{k,p} (\Omega)$ consisting of all functions which can be approximated in the $W^{k,p}$-norm by smooth functions of compact support. When $p=2$, we write $H^k(\Omega)$ in place of $W^{k,2}(\Omega)$, and $H^k_0 (\Omega)$ in place of $W^{k,2}_0(\Omega)$.
\end{definition}
(When $k=0$, we recover $W^{0,p}(\Omega) = W^{0,p}_0 (\Omega) = L^p(\Omega)$.) We regard $W^{k,p}_0 (\Omega)$ as the set of $W^{k,p}$-functions which vanish up to order $k-1$ on $\partial\Omega$ (this may be made more precise using the notion of trace operators, see below).

\begin{theorem}
Let $\Omega \subset \R^d$ be a bounded domain with sufficiently smooth boundary. For all $k \geq 0$ and $1\leq p\leq \infty$, $W^{k,p}(\Omega)$ is a Banach space, and $C^\infty (\overline{\Omega})$ is dense in $W^{k,p}(\Omega)$. Moreover, $W^{k,p}_0(\Omega)$ is a closed subspace (and thus also a Banach space). When $k=2$, $H^k(\Omega)$ and $H^k_0 (\Omega)$ are Hilbert spaces with respect to the canonical inner product
\begin{displaymath}
	\langle u,v\rangle_{H^k(\Omega)} = \langle u,v\rangle_{L^2(\Omega)} + \sum_{1\leq |\alpha|\leq k} \left\langle \frac{\partial^\alpha u}
	{\partial x^\alpha}, \frac{\partial^\alpha v}{\partial x^\alpha}\right\rangle_{L^2(\Omega)}
\end{displaymath}
for $u,v \in H^k(\Omega)$ or $H^k_0 (\Omega)$, as appropriate.
\end{theorem}

Our $V$ will be either $H^1(\Omega)$ or $H^1_0(\Omega)$; for future reference and comparison, we state explicitly what the inner product is:
\begin{equation}
\label{eq:inner-product-h1omega}
	\langle u,v\rangle_{H^1(\Omega)} = \int_\Omega u\,\overline{v}\,\textrm{d}x + \int_\Omega \nabla u\cdot\overline{\nabla v}\,\textrm{d}x
\end{equation}
for all $u,v \in H^1(\Omega)$ (or $H^1_0(\Omega)$). Note that in the case of $H^1_0(\Omega)$, the alternative expression $\int_\Omega \nabla u\cdot\overline{\nabla v}\,\textrm{d}x$ already defines an inner product, as follows from \emph{Friedrichs' inequality}. We will not need this.

If $u \in H^1(\Omega)$, even though $u$ is \emph{a priori} only defined almost everywhere, we can still give meaning to $u|_{\partial\Omega} \in L^2(\partial\Omega)$ via the \emph{trace theorem}: there exists a unique continuous linear operator $\trace : H^1(\Omega) \to L^2(\partial\Omega)$ such that if $u \in H^1(\Omega) \cap C(\overline{\Omega})$, then $\trace u = u|_{\partial\Omega}$ is just the restriction $u|_{\partial\Omega}$ of $u$ to $\partial\Omega$. Since this operator is bounded, we have the following \emph{trace inequality}: there exists a constant $C=C(\Omega) > 0$ such that
\begin{equation}
\label{eq:trace-omega}
	\|\trace u\|_{L^2(\partial\Omega)} \leq C \|u\|_{H^1(\Omega)}
\end{equation}
for all $u \in H^1(\Omega)$. If $\Omega$ has sufficiently smooth boundary, then the function $u$ is in $H^1_0 (\Omega)$ if and only if $u \in H^1(\Omega)$ and $\trace u = 0$ in $L^2(\partial\Omega)$. Loosely, the trace allows us to act as if $u$ had well-defined values on $\partial\Omega$ even though $\partial\Omega$ has measure zero; for this reason, often we just write $u$ in place of $\trace u$.

\begin{definition}
\label{def:delta-u}
Let $u \in H^1(\Omega)$. We say that $u$ admits a Laplacian $\Delta u$ in the weak sense if all weak first-order partial derivatives of $u$ exist, and if there exists a measurable function $f =: \Delta u$ such that the following Green's identity holds:
\begin{equation}
\label{eq:laplacian-def}
	\int_\Omega f \phi \,\textrm{d}x = -\int_\Omega \nabla u \cdot \nabla \phi\,\textrm{d}x \left(= \int_\Omega u\Delta \phi\,\textrm{d}x\right)
\end{equation}
for all $\phi \in C_c^\infty (\Omega)$ (and hence all $\phi \in H^1_0 (\Omega)$, by density). In this case we always write $\Delta u$ in place of $f$.
\end{definition}

Observe that if $u$ is a smooth function, then, by the divergence theorem,
\begin{displaymath}
	\int_\Omega \Delta u \phi + \nabla u \cdot \nabla \phi\,\textrm{d}x = \int_\Omega \divergence (\phi \nabla u)\,\textrm{d}x = 0
\end{displaymath}
for all $\phi \in C_c^\infty (\Omega)$, since $\phi = 0$ in a neighbourhood of the boundary. This identity motivates our definition: the Laplacian is the function $f$, if one exists, such that \eqref{eq:laplacian-def} holds. Our definition is also consistent with \eqref{eq:weak-derivative-1} since it can be obtained by applying \eqref{eq:weak-derivative-1} to each of the functions $\frac{\partial u}{\partial x_j}$ in place of $u$, and summing over $j$.



If $u \in H^1(\Omega)$ admits a Laplacian $\Delta u$ in the weak sense and $\Delta u \in L^2(\Omega)$, then we can also give meaning to the \emph{outer normal derivative} $\frac{\partial u}{\partial\nu}$ via a weak version of the divergence theorem.

\begin{definition}
Let $u \in H^1(\Omega)$ such that $\Delta u \in L^2(\Omega)$. We say that $f \in L^2(\partial\Omega)$ is a (in fact, \emph{the}) outer normal derivative of $u$ at $\partial\Omega$, and write $f = \frac{\partial u}{\partial \nu}$, if
\begin{equation}
\label{eq:outer-derivative-omega}
	\int_\Omega \nabla u \cdot \nabla v + v\Delta u \,\textrm{d}x = \int_{\partial\Omega} fv\,\textrm{d}s
\end{equation}
for all $v \in H^1({\Omega})$, where in the boundary integrand $v$ is really the trace $\trace v$ of $v \in H^1(\Omega)$.
\end{definition}
Note that if $u \in C^2(\overline{\Omega})$, then \eqref{eq:outer-derivative-omega} is an immediate consequence of the divergence theorem, since
\begin{displaymath}
	\int_\Omega \nabla u \cdot \nabla v + v\Delta u \,\textrm{d}x = \int_\Omega \divergence (v\nabla u)\,\textrm{d}x;
\end{displaymath}
for less regular $u$, \eqref{eq:outer-derivative-omega} is taken as a \emph{definition}.

\begin{theorem}[Embedding theorems, small selection]
\label{thm:embedding-omega}
Let $\Omega \subset \R^d$ be a bounded domain with sufficiently smooth boundary, and let $0\leq j \leq k$ and $1\leq q \leq p \leq \infty$. Then $W^{k,p}(\Omega)$ embeds continuously and densely in $W^{j,q} (\Omega)$, and $W^{k,p}_0 (\Omega)$ embeds continuously and densely in $W^{j,q}_0 (\Omega)$. The embeddings are compact if $k > j$ or if $p > q$.

In particular, $H^1(\Omega)$ and $H^1_0 (\Omega)$ are compactly and densely embedded in $L^2(\Omega)$.
\end{theorem}

To define Laplace-type operators on $\Omega$ (plus suitable boundary conditions on $\partial\Omega$), we use the form approach of Section~\ref{sec:forms}: we define a form
\begin{equation}
\label{eq:aq-omega}
	a(u,v) = \int_\Omega \nabla u\cdot\overline{\nabla v}\,\textrm{d}x,
\end{equation}
which makes sense whenever $u,v \in H^1(\Omega)$.

\begin{lemma}
\label{lem:aq-omega}
Let $\Omega \subset \R^d$ be a bounded Lipschitz domain. The form $a$ is a bounded Hermitian form which is $L^2$-elliptic on $H^1(\Omega) \times H^1(\Omega)$, and coercive on $H^1_0 (\Omega) \times H^1_0 (\Omega)$.
\end{lemma}

We denote by $-\Delta^N: D(-\Delta^N) \subset L^2(\Omega) \to L^2(\Omega)$ the operator associated with $a$ on $H^1(\Omega)$, which we call the \emph{Neumann Laplacian} (or \emph{Laplacian with Neumann boundary conditions}), and by $-\Delta^D : D(-\Delta^D) \subset L^2(\Omega) \to L^2(\Omega)$ the operator associated with $a$ on $H^1_0(\Omega)$, which we call the \emph{Dirichlet Laplacian} (or \emph{Laplacian with Dirichlet boundary conditions}).


\begin{theorem}
\label{thm:laplacians-omega}
\begin{enumerate}
\item[\emph{\textbf{(a)}}] The operator $-\Delta^D$ is given by
\begin{displaymath}
\begin{aligned}
	D(-\Delta^D) &= \big\{u \in H^1_0(\Omega): \Delta u \in L^2(\Omega) \text{ (as a weak derivative)}\big\}.\\
	-\Delta^D u &= -\Delta u.
\end{aligned}
\end{displaymath}
\item[\emph{\textbf{(b)}}] The operator $-\Delta^N$ is given by
\begin{displaymath}
\begin{aligned}
	D(-\Delta^N) &= \left\{u \in H^1_0(\Omega): \Delta u \in L^2(\Omega),\, \frac{\partial u}{\partial\nu} \in L^2(\partial\Omega)
	\text{ (as weak derivatives)}\right\}.\\
	-\Delta^N u &= -\Delta u.
\end{aligned}
\end{displaymath}
\end{enumerate}
\end{theorem}

The proof involves combining the definition of the associated operator with the properties of $H^1(\Omega)$-functions listed above. See Exercises 1.4 and 1.5. The spectral properties of $-\Delta^D$ and $-\Delta^N$ -- the foundational properties of most interest to us -- follow immediately from Lemma~\ref{lem:aq-omega}, Theorem~\ref{thm:embedding-omega} and Theorem~\ref{thm:spectral} (the spectral theorem).

\begin{theorem}[Spectral theorem for the Dirichlet and Neumann Laplacians]
\label{thm:spectral-omega}
Let $\Omega \subset \R^d$ be a bounded Lipschitz domain. The operators $-\Delta^D$ and $-\Delta^N$ are densely defined self-adjoint operators with compact resolvent. In particular, their respective spectra take the form
\begin{displaymath}
\begin{aligned}
	-\Delta^D: \qquad & \lambda_1 \leq \lambda_2 \leq \lambda_3 \leq \ldots \to \infty,\\
	-\Delta^N: \qquad & \mu_1 \leq \mu_2 \leq \mu_3 \leq \ldots \to \infty.
\end{aligned}
\end{displaymath}
The corresponding eigenfunctions $\psi_k \sim \lambda_k$ and $\varphi_k \sim \mu_k$ may, respectively, be chosen to form a real orthonormal basis of $L^2(\Omega)$. The eigenvalues admit the min-max and max-min characterisations of Theorem~\ref{thm:courant-fischer}, with Rayleigh quotient
\begin{displaymath}
	\frac{a(u,u)}{\|u\|_2^2} = \frac{\int_\Omega |\nabla u|^2\,\textrm{d}x}{\int_\Omega |u|^2\,\textrm{d}x}.
\end{displaymath}
\end{theorem}

In particular, since the eigenvalues are real and the eigenfunctions may be chosen real, we can, and will, restrict ourselves to real-valued functions and real spaces throughout.

Actually, one can say more: in each case the first eigenvalue is \emph{simple}, $\lambda_1 < \lambda_2$ and $\mu_1 < \mu_2$, and $\psi_1$ and $\varphi_1$ may be chosen strictly positive everywhere in $\Omega$. (This is a consequence of a more general theory, Perron--Frobenius theory.) It is a consequence of \emph{Friedrichs' inequality} (equivalently, the coercivity of $a$ on $H^1_0 (\Omega)$) that $\lambda_1 > 0$, while $\mu_1 = 0$ with the corresponding eigenfunctions, the solutions of
\begin{displaymath}
\begin{aligned}
	-\Delta u &= 0 \cdot u \qquad &&\text{in } \Omega,\\ \frac{\partial u}{\partial\nu} &=0 \qquad &&\text{on } \partial\Omega,
\end{aligned}
\end{displaymath}
being just the constant functions.

When $d=1$ and $\Omega=(0,1)$,
\begin{displaymath}
	\lambda_k = \pi^2k^2 \qquad \text{with} \qquad \psi_k (x) = c\sin (\pi kx),
\end{displaymath}
$c \neq 0$, and similarly
\begin{displaymath}
	\mu_k = \pi^2(k-1)^2, \qquad \varphi_k(x) = c \cos (\pi(k-1)x).
\end{displaymath}

\subsection{Application: the heat equation}
\label{sec:heat}

Let $f \in L^2(\Omega)$ be an initial condition, representing an initial distribution of heat, or some chemical concentration, or a population distribution, or some other quantity whose evolution will be governed by a diffusion process in a homogeneous medium, at time $t=0$. Denote by $u=u(t,x)$ the corresponding concentration of heat/the chemical/population etc.\ at $x \in \Omega$ at time $t>0$, then $u$ will be a solution of the heat equation
\begin{displaymath}
\begin{aligned}
	\frac{\partial u}{\partial t} &= \Delta u \qquad &&\text{in } (0,\infty) \times \Omega,\\
	u &= 0 \text{ (or $\tfrac{\partial u}{\partial\nu}=0$)} \qquad &&\text{on } (0,\infty) \times \partial\Omega,\\
	u(0,x) &=f
\end{aligned}
\end{displaymath}
(where the Laplacian $\Delta u = \sum_{j=1}^d \frac{\partial^2 u}{\partial x^2}$ is taken with respect to the spatial variables). Formally, the solution is given by an abstract Fourier series: since $\psi_k$, $\varphi_k$ form orthonormal bases of $L^2(\Omega)$, we may write
\begin{displaymath}
\begin{aligned}
	f(x) &= \sum_{k=1}^\infty \langle f,\psi_k\rangle_{L^2(\Omega)} \psi_k (x)\\
	&= \sum_{k=1}^\infty \langle f,\varphi_k\rangle_{L^2(\Omega)} \varphi_k (x)
\end{aligned}
\end{displaymath}
almost everywhere in $\Omega$. In the Dirichlet case the solution $u$ is given by the series
\begin{displaymath}
	u(t,x) = e^{\Delta^D} f(x) = \sum_{k=1}^\infty e^{-\lambda_k t} \langle f,\psi_k\rangle_{L^2(\Omega)} \psi_k(x),
\end{displaymath}
and in the Neumann case by
\begin{displaymath}
	u(t,x) = e^{\Delta^N} f(x) = \sum_{k=1}^\infty e^{-\mu_k t} \langle f,\varphi_k\rangle_{L^2(\Omega)} \varphi_k(x).
\end{displaymath}
In either case, the solution may alternatively be represented using the respective \emph{heat kernel}; for example, in the Dirichlet case, this reads
\begin{displaymath}
	k_t (x,y) = \sum_{k=1}^\infty e^{-\lambda_k t} \psi_k(x)\psi_k(y),
\end{displaymath}
and in this case
\begin{displaymath}
	u(t,x) = \int_\Omega k_t(x,y)f(y)\,\textrm{d}y = k_t(x,\cdot) \ast f.
\end{displaymath}
Loosely, $\lambda_1$ and $\mu_2$ determine the rate of convergence/heat loss from $\Omega$. In the Dirichlet case, where there is ``perfect cooling'' at the boundary, since $\lambda_1 > 0$,
\begin{displaymath}
	\|u_t\|_2 \leq e^{-\lambda_1 t}\|f\|_2 \to 0
\end{displaymath}
exponentially as $t\to \infty$, with exponent (at least) $\lambda_1$. In the Neumann case, which represents perfect insulation, i.e.\ no heat loss through the boundary, since $\mu_1 = 0$ and $\mu_2 > 0$,
\begin{displaymath}
	\sum_{k=1}^\infty e^{-\mu_k t}\langle f,\varphi_k\rangle \to \langle f,\varphi_1\rangle\varphi_1 = \frac{1}{|\Omega|} \int_\Omega f\,\textrm{d}x:
\end{displaymath}
the equilibrium profile is the constant function, i.e. a uniform distribution of heat, and the rate of convergence is essentially determined by $\mu_2$.

Only for very few domains $\Omega$ can the eigenvalues and eigenfunctions be computed explicitly; and if we introduce more complicated operators such as Schr{\"o}dinger operators of the form $-\Delta+q$, where $q$ is a suitable function on $\Omega$, or more involved boundary conditions such as Robin boundary conditions, then there may be no domains $\Omega$ at all for which the eigenvalues are explicitly known. See Exercises 1.4 and 1.5.
\bigskip

\textbf{Goals:}
\begin{itemize}
\item Construct analogues of these operators on metric graphs (``quantum graphs'');
\item Obtain a spectral theorem for them;
\item Study how the spectrum is related to the geometry of the graph.
\end{itemize}

\section{Quantum graphs}
\label{chapter:quantum-graphs}

\subsection{Metric graphs}
\label{sec:metric-graphs}

Idea: a graph $\Gamma = (V,E)$ consists of:
\begin{enumerate}
\item A \emph{vertex set} $V$ of points, usually assumed countable, $V=\{v_1,v_2,\ldots\}$;
\item An \emph{edge set} $E$, also generally taken countable, $E=\{e_1,e_2,\ldots\}$.
\end{enumerate}
Each edge is $e$ associated \emph{(incident)} with two vertices $v_1,v_2$; we will write $e \sim v_1 v_2$. Two vertices are \emph{adjacent} if there is an edge incident with both. The \emph{degree} of a vertex is the number of edges incident with it.\footnote{Technical note: if $e$ is a \emph{loop} at $v$, $e\sim v v$, which in general we permit, then $e$ counts twice to the degree of $v$.} Here we will always assume:

\begin{assumption}
$V$ and $E$ are \emph{finite} sets: there exist $m,n \in \N$ such that $\# V = n$, $\# E = m$.
\end{assumption}

There are various different ``kinds'' of graphs depending on the nature of the edges. Most notable is the distinction between \emph{discrete graphs} and \emph{metric graphs}. We will almost always only consider metric graphs.\medskip

\textbf{Discrete graphs:} $E$ constitutes a set of adjacency (binary) relations between vertices: $v_1 \sim v_2$ if there exists an edge $e$ such that $e \sim v_1v_2$. The edges have no meaning beyond this. We define functions (really vectors) $f : V \to \C$, thus there is only one function space, isomorphic to $\C^n$ (where the $k$th row/column corresponds to the $k$th vertex).

Operators on these function spaces correspond to matrices; for example, the \emph{adjacency matrix} $A\in \C^{n \times n}$ is defined by $A_{ij} = 1$ if there is an edge running from $v_i$ to $v_j$, and $0$ otherwise. There are also (at least) two commonly studied matrices known as (discrete) Laplacians; perhaps the most common is defined as $L=D-A$, where $A$ is the adjacency matrix and $D$ is the diagonal matrix whose $(i,i)$-entry equal to the degree of vertex $v_i$ (see Definition~\ref{def:on-the-edges} and note that for discrete graphs this definition is the same).

Discrete graphs, and their Laplacians, have, in general, been studied far more intensively over the years than metric graphs; many of the topics of these notes have much more well-developed pendants on discrete graphs. One of the best known references, and a natural starting point, is the book \cite{chung}.\medskip

\textbf{Metric graphs:} Each edge is identified with an interval, which are ``glued together'' at their endpoints in accordance with the incidence relations.
\begin{assumption}
All edges have finite length: each edge $e$ may be identified with a compact interval $[0,\ell_e]$ for some $\ell_e > 0$.
\end{assumption}

We can also distinguish between \emph{directed} and \emph{undirected} graphs. Roughly, $\Gamma$ is directed if the edge has an orientation, that is, the order of the vertices in the incidence relation is important: $e \sim v_1v_2$ is different from $e \sim v_2 v_1$. It is undirected if there is no distinction and edges impose a symmetric relation on the vertices.

This distinction is more directly visible in the world of discrete graphs; for example, the adjacency matrix of the graph is symmetric if the graph is undirected and not generally symmetric if it is directed. (Another example: X, formerly Twitter, with its network of followers, defines a directed discrete graph whose vertices are the users; Facebook, with its network of ``friends'', defines an undirected discrete graph.) Our metric graphs will be undirected; in other words, the choice of the orientation of the intervals will not matter.

\emph{Formally:} (Roughly following \cite{mugnolo}.) For $k=1,\ldots,m$, take positive real numbers $\ell_1,\ldots,\ell_k > 0$ (which will be the edge lengths), and for each $k$ take an interval of length $\ell_k$, $I_k := [x_k,y_k]$, so that $y_k - x_k = \ell_k$. Set
\begin{displaymath}
	G := \bigsqcup_{k=1}^m I_k
\end{displaymath}
to be their (formal) disjoint union, as well as $S:= \{x_k,y_k\}_{k=1}^m$ to be the endpoint set, treated as a set with $2m$ elements even if some elements coincide as real numbers. Define a \emph{partition} $P = \{V_1,\ldots,V_n\}$ of $S$:
\begin{displaymath}
	V_j \neq \emptyset, \qquad V_i \cap V_j = \emptyset \text{ for } i \neq j, \qquad \bigcup_{j=1}^n V_j = S.
\end{displaymath}
Impose an equivalence relation on $G$: given $x,y \in G$,
\begin{displaymath}
	x \sim y \text{ iff } \begin{cases} x,y \in I_k \text{ for some } k \text{ and } x=y \text{ in } I_k, \text{or}\\
	x,y \in V_j \subset S \text{ for some } j=1,\ldots,n.\end{cases}
\end{displaymath}
\begin{definition}
\label{def:on-the-edges}
We define the \emph{edge set} $E = \{e_1,\ldots,e_m\}$ by
\begin{displaymath}
	e_k := I_k/\!\sim\,\, = \{ [x] : [x] \text{ has a representative belonging to } I_k\}, \qquad k=1,\ldots,m,
\end{displaymath}
and the \emph{vertex set} $V = \{v_1,\ldots,v_n\}$ by
\begin{displaymath}
	v_j := V_j/\!\sim, \qquad j=1,\ldots,n.
\end{displaymath}
We call $\Gamma = (V,E)$ a \emph{metric graph}. We say $e_k$ is \emph{incident} with $v_j$ if there exists $x \in V_j \cap I_k$, and we say $v_i$ and $v_j$ are \emph{adjacent} if there exists $k$ such that, for $I_k = [x_k,y_k]$, either $x_k \in V_i$ and $y_k \in V_j$, or $x_k \in V_j$ and $y_k \in V_i$.\footnote{This assumption formally guarantees that the graph is undirected. Also, the case where the edge $e_k$  loop is covered by the possibility that $V_i = V_j$ and thus $v_i=v_j$.} 
 The \emph{degree} of the vertex $v_k$, $\deg v_k$, is the cardinality of $V_k$.
\end{definition}

In practice we identify $\Gamma$ with $G/\!\sim$ and treat it as a set of points. In particular, given an edge $e_k = I_k/\!\sim$ with incident vertices $v_{k,1}$ and $v_{k,2}$, we may, and usually will, introduce canonical ``local coordinates'' of the form $e_k \simeq [0,\ell_k]$, and either ($0 \sim v_{k,1}$ and $\ell_k \sim v_{k,2})$, or ($0\sim v_{k,2}$ and $\ell_k \sim v_{k,1}$). We will not distinguish between $x \in e_k$, $x \in \Gamma$, $x \in I_k$, and $x \in [0,\ell_k]$, except where strictly notationally necessary.

We may alternatively think of the local coordinates on each edge as a ``chart'' and the collection of charts as an ``atlas'', akin to how we describe manifolds. Indeed, it is sometimes useful to think of a metric graph as a one-dimensional manifold with singularities, the vertices.

\begin{example}
\label{ex:3-star}
For $m=3$, take $I_k = [0_k,1_k]$, $k=1,2,3$, then $S=\{0_1,0_2,0_3,1_1,1_2,1_3\}$.
\begin{figure}[H]
\centering
\begin{tikzpicture}
\draw[thick,red] (-3,0) -- (-1,0);
\draw[thick,purple] (0,0) -- (2,0);
\draw[thick,blue] (3,0) -- (5,0);
\draw[fill] (-3,0) circle (1.5pt);
\draw[fill] (-1,0) circle (1.5pt);
\draw[fill] (0,0) circle (1.5pt);
\draw[fill] (2,0) circle (1.5pt);
\draw[fill] (3,0) circle (1.5pt);
\draw[fill] (5,0) circle (1.5pt);
\node at (-3,0) [anchor=north] {${\color{red} 0_1}$};
\node at (-1,0) [anchor=north] {${\color{red} 1_1}$};
\node at (0,0) [anchor=north] {${\color{purple} 0_2}$};
\node at (2,0) [anchor=north] {${\color{purple} 1_2}$};
\node at (3,0) [anchor=north] {${\color{blue} 0_3}$};
\node at (5,0) [anchor=north] {${\color{blue} 1_3}$};
\node at (-2,0) [anchor=south] {$I_1$};
\node at (1,0) [anchor=south] {$I_2$};
\node at (4,0) [anchor=south] {$I_3$};
\end{tikzpicture}
\end{figure}

If we set $V_1 = \{0_1\}$, $V_2 = \{0_2\}$, $V_3 = \{1_1\}$, $V_4 = \{1_1,1_2,0_3\}$, then we generate a \emph{star graph} on three edges: see Figure~\ref{fig:3-star}.
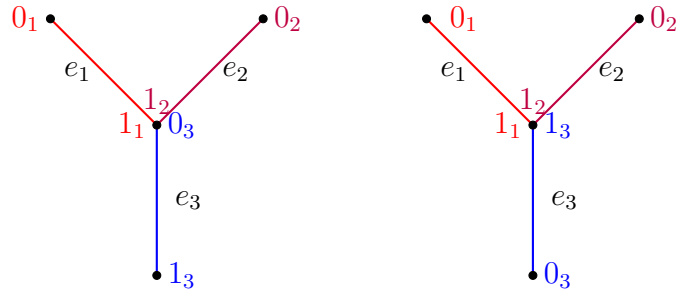
\begin{figure}[H]
\centering
\begin{tikzpicture}
\draw[thick,red] (-1.414,1.414) -- (0,0);
\draw[thick,purple] (1.414,1.414) -- (0,0);
\draw[thick,blue] (0,0) -- (0,-2);
\draw[fill] (-1.414,1.414) circle (1.5pt);
\draw[fill] (1.414,1.414) circle (1.5pt);
\draw[fill] (0,0) circle (1.5pt);
\draw[fill] (0,-2) circle (1.5pt);
\node at (-1.414,1.414) [anchor=east] {${\color{red} 0_1}$};
\node at (0,0) [anchor=east] {${\color{red} 1_1}$};
\node at (1.414,1.414) [anchor=west] {${\color{purple} 0_2}$};
\node at (0,0) [anchor=south] {${\color{purple} 1_2}$};
\node at (0,0) [anchor=west] {${\color{blue} 0_3}$};
\node at (0,-2) [anchor=west] {${\color{blue} 1_3}$};
\node at (-0.727,0.707) [anchor=east] {$e_1$};
\node at (0.727,0.707) [anchor=west] {$e_2$};
\node at (0.1,-1) [anchor=west] {$e_3$};
\draw[thick,red] (3.586,1.414) -- (5,0);
\draw[thick,purple] (6.414,1.414) -- (5,0);
\draw[thick,blue] (5,0) -- (5,-2);
\draw[fill] (3.586,1.414) circle (1.5pt);
\draw[fill] (6.414,1.414) circle (1.5pt);
\draw[fill] (5,0) circle (1.5pt);
\draw[fill] (5,-2) circle (1.5pt);
\node at (4.414,1.414) [anchor=east] {${\color{red} 0_1}$};
\node at (5,0) [anchor=east] {${\color{red} 1_1}$};
\node at (6.414,1.414) [anchor=west] {${\color{purple} 0_2}$};
\node at (5,0) [anchor=south] {${\color{purple} 1_2}$};
\node at (5,0) [anchor=west] {${\color{blue} 1_3}$};
\node at (5,-2) [anchor=west] {${\color{blue} 0_3}$};
\node at (4.273,0.707) [anchor=east] {$e_1$};
\node at (5.727,0.707) [anchor=west] {$e_2$};
\node at (5.1,-1) [anchor=west] {$e_3$};
\end{tikzpicture}
\caption{The $3$-star graph of Example~\ref{ex:3-star} with the first-mentioned choice of coordinates (left) and the alternative coordinates (right).}
\label{fig:3-star}
\end{figure}

Choosing, for example, $V_1 = \{0_1\}$, $V_2 = \{0_2\}$, $V_3 = \{0_1\}$, $V_4 = \{1_1,1_2,1_3\}$ would lead to the same graph, but the orientation of $e_3$ in the chosen ``local coordinates'' would be reversed.
\end{example}

We can define a metric on $\Gamma$ in the obvious way: choosing, for each edge $e_k \in E$, a coordinate representation, we define the natural topology on $\Gamma$ induced by the Euclidean metric on each edge:

\begin{definition}
\label{def:topology}
A set $\Omega \subset \Gamma$ is \emph{open} iff $\Omega \cap e_k$ is relatively open in $e_k$ for all $k=1,\ldots,m$.
\end{definition}

In particular, if some vertex $v \in \Omega$ and $e \sim v$, where $e$ has local coordinates $[0,\ell]$ with $0 \sim v$, then, since $0 \in \Omega$, there must exist some relatively open interval $[0,a)$, $a \in (0,\ell]$, contained in $\Omega \cap e$.

It is a very elementary but rather tedious exercise to show that this topology is independent of the choice of local coordinates on each edge (including the orientation of the edges), as well as of the labelling of the edges.

This topology is canonically metrisable.

\begin{definition}
Let $\Gamma$ be a metric graph.
\begin{enumerate}
\item[\textbf{(a)}] A \emph{path} $P$ in $\Gamma$ is (the image in $\Gamma$ of) a map $\phi: [0,1] \to \Gamma$ which is injective on $[0,1)$ and continuous with respect to the topology induced by the open sets of Definition~\ref{def:topology}. If $\phi(0)=x$ and $\phi(1)=y$, then we say $P$ is a path from $x$ to $y$.
\item[\textbf{(b)}] The path $P$ is a \emph{closed path} if $\phi(0)=\phi(1)$.
\end{enumerate}
\end{definition}

In general, if $\phi(0)=x$, $\phi(1)=y$, then there exists a (possibly empty) sequence of distinct vertices $v_1,v_2,\ldots,v_k$ lying on $P$ such that $v_1 \sim v_2 \sim \ldots \sim v_k$, $x$ is on an edge $e_x$ incident with $v_1$, and $y$ is on an edge $e_y$ incident with $v_k$. We will write $e_j$ for the edge $\sim v_j v_{j+1}$.
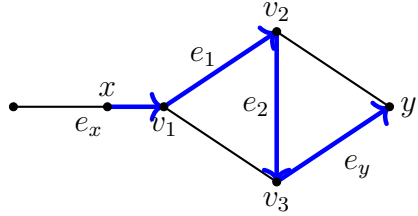
\begin{figure}[H]
\centering
\begin{tikzpicture}
\draw[thick] (-2,0) -- (0,0);
\draw[thick] (0,0) -- (1.5,1);
\draw[thick] (0,0) -- (1.5,-1);
\draw[thick] (1.5,1) -- (1.5,-1);
\draw[thick] (1.5,1) -- (3,0);
\draw[thick] (1.5,-1) -- (3,0);
\draw[line width=0.6mm,blue,->] (-0.75,0) -- (0,0);
\draw[line width=0.6mm,blue,->] (0,0) -- (1.5,1);
\draw[line width=0.6mm,blue,->] (1.5,1) -- (1.5,-1);
\draw[line width=0.6mm,blue,->] (1.5,-1) -- (3,0);
\draw[fill] (-2,0) circle (1.5pt);
\draw[fill] (-0.75,0) circle (1.5pt);
\draw[fill] (0,0) circle (1.5pt);
\draw[fill] (1.5,1) circle (1.5pt);
\draw[fill] (1.5,-1) circle (1.5pt);
\draw[fill] (3,0) circle (1.5pt);
\node at (-0.75,0) [anchor=south] {$x$};
\node at (0,0) [anchor=north] {$v_1$};
\node at (1.5,1) [anchor=south] {$v_2$};
\node at (1.5,-1) [anchor=north] {$v_3$};
\node at (3,0) [anchor=west] {$y$};
\node at (-1,0) [anchor=north] {$e_x$};
\node at (2.25,-0.5) [anchor=north west] {$e_y$};
\node at (0.85,0.4) [anchor=south east] {$e_1$};
\node at (1.55,0) [anchor=east] {$e_2$};
\end{tikzpicture}
\end{figure}

\begin{definition}
\label{def:distance}
Let $\Gamma$ be a metric graph.
\begin{enumerate}
\item[\textbf{(a)}] The \emph{length} of an edge $e \simeq [0,\ell]$ is $\ell$, we will also write $|e| := \ell$.
\item[\textbf{(b)}] The \emph{local} or \emph{edgewise distance function} on an edge $e \simeq [0,\ell]$ is defined by
\begin{displaymath}
	\dist_e (x,y) = |x-y| \qquad \text{in the interval } [0,\ell], \text{ if } x,y \in e.
\end{displaymath}
\item[\textbf{(c)}] The \emph{length} of a path $P$ in $\Gamma$, with the notation described above, is
\begin{displaymath}
	|P| := \dist_{e_x} (x,v_1) + \sum_{j=1}^{k-1} |e_j| + \dist_{e_y} (v_k,y).
\end{displaymath}
\item[\textbf{(d)}] The \emph{distance} between $x,y \in \Gamma$ is defined as
\begin{displaymath}
	\dist (x,y) \equiv \dist_\Gamma (x,y) = \inf \{|P| : P \text{ is a path from $x$ to $y$}\} \in [0,\infty].
\end{displaymath}
\item[\textbf{(e)}] $\Gamma$ is \emph{connected} if $\dist (x,y)< \infty$ for all $x,y \in \Gamma$.
\end{enumerate}
\end{definition}

\begin{theorem}
The above definitions, in particular the edge lengths and the distance function, do not depend on the choice of local coordinates, nor the labelling of the edges. In particular, $\dist$ does define a metric on $\Gamma$, which induces the same topology as in Definition~\ref{def:topology}. For this topology, $\Gamma$ is a complete metric space, which is compact under our assumption that $E$ is finite and each edge has finite length.
\end{theorem}

\begin{proof}
Tedious exercise.
\end{proof}

A metric graph $\Gamma$ consisting of finitely many edges of finite length is thus generally called a \emph{compact metric graph} in the literature, cf.~\cite[Definition~1.3.4]{berkolaiko}. Finally, we emphasise that we will generally treat $\Gamma$ as a metric space of points $x \in \Gamma$ rather than the pair $(V,E)$ of a vertex set and an edge set (although whenever convenient we will identify the two).

\subsection{Function spaces}
\label{sec:function-spaces}

Denote by $\mu$ Lebesgue measure on $\R$, which induces a measure on each edge of a metric graph $\Gamma$ via any choice of local coordinates. Define Lebesgue measure $\mu$ on $\Gamma$ as a direct sum of Lebesgue measure on the intervals: $\Omega \subset \Gamma$ is measurable iff $\Omega \cap e_k$ is measurable for all $k=1,\ldots,m$, and in this case
\begin{displaymath}
	\mu (\Omega) = \sum_{k=1}^m \mu (\Omega \cap e_k).
\end{displaymath}
It is another easy exercise to check that every path is measurable, and its length corresponds with its Lebesgue measure. Equipped with $\dist$ and $\mu$, $\Gamma$ is a \emph{metric measure space}.

We can introduce integrable functions accordingly: for $f: \Gamma \to \R$,
\begin{equation}
\label{eq:integral}
	\int_\Gamma f\,\textrm{d}\mu \equiv \int_\Gamma f(x)\,\textrm{d}x = \sum_{k=1}^m \int_{e_k} f(x)\,\textrm{d}x
	\simeq \sum_{k=1}^m \int_0^{\ell_k} f(x)\,\textrm{d}x
\end{equation}
in local coordinates; in particular, we identify $L^p(e_k)$ with $L^p([0,\ell_k])$. We also write $f|_e$ for the restriction of a function $f : \Gamma \to \R$ to $e$.

\begin{definition}
For $1\leq p \leq \infty$, we set
\begin{displaymath}
	L^p (\Gamma) := \{ f: \Gamma \to \C: f|_{e_k} \in L^p(e_k), \, k=1,\ldots,m\} \simeq \bigoplus_{k=1}^m L^p ([0,\ell_k]).
\end{displaymath}
\end{definition}

(If $\#E = \infty$, this definition would have to be modified.)

\begin{theorem}
The quantity
\begin{equation}
\label{eq:graph-norm}
	\|f\|_p = \begin{cases} \left(\sum_{k=1}^m \int_{e_k} |f(x)|^p\,\textrm{d}x\right)^{1/p}, \qquad & 1 \leq p < \infty,\\
	\max_{k=1,\ldots,m} \|f|_{e_k}\|_{L^\infty (e_k)} = \esssup_{x \in \Gamma} |f(x)|, \qquad & p=\infty,\end{cases}
\end{equation}
defines a norm on $L^p (\Gamma)$; equipped with this norm $L^p(\Gamma)$ is a Banach space. When $p=2$ it is a Hilbert space for
\begin{displaymath}
	\langle f,g\rangle = \int_\Gamma f\overline{g}\,\textrm{d}x \equiv \sum_{k=1}^m \int_{e_k} f(x)\overline{g(x)}\,\textrm{d}x.
\end{displaymath}
\end{theorem}

$L^p$-spaces do not ``see'' the topology of the graph: for that we need continuous functions.

\begin{definition}
We set
\begin{displaymath}
\begin{aligned}
	C(\Gamma) := \{ f : \Gamma \to \C: & \, f|_{e_k} \in C(e_k) \text{ for all $k=1,\ldots,m$ and}\\
	&\text{whenever $x,y \in V_j$ for some vertex $v_j$, we have } f(x)=f(y)\}.
\end{aligned}
\end{displaymath}
\end{definition}

In particular: $f$ has a well-defined value $f(v_j)$ at each vertex $v_j \in V$, $j=1,\ldots,n$. Moreover:

\begin{lemma}
$f \in C(\Gamma)$ iff $f$ is continuous with respect to the metric/topology of Section~\ref{sec:metric-graphs}; in particular, if $v \sim e$ and $x_n \in e$ with $x_n \to v$ in $\Gamma$, then $f(x_n) \to f(v)$.
\end{lemma}

\begin{theorem}
Under our assumptions on $\Gamma$ (i.e. for $\Gamma$ a compact metric graph), the quantity
\begin{displaymath}
	\|f\|_\infty = \max_{x \in \Gamma} |f(x)|
\end{displaymath}
defines a norm on $\Gamma$ with respect to which $C(\Gamma)$ is a Banach space.
\end{theorem}

To define differential operators on $\Gamma$, we need Sobolev spaces. Here we can mostly make do with first-order Sobolev spaces in $\R$: we recall that, for an interval $I \subset \R$ and $1 \leq p \leq \infty$,
\begin{displaymath}
	W^{1,p}(I) = \{f \in L^p(I): f \text{ has a weak derivative } f' \in L^p(I)\}.
\end{displaymath}
The main challenge in introducing analogous Sobolev spaces on $\Gamma$ is determining what to do at the vertices.

\begin{lemma}[Embedding theorem]
\label{lem:embedding-1d}
Let $I \subset \R$ be a bounded interval, $1\leq p \leq \infty$. Then there exists a continuous injection
\begin{displaymath}
	W^{1,p}(I) \hookrightarrow C(\overline{I}).
\end{displaymath}
\end{lemma}

We may thus characterise $W^{1,p}$-functions as continuous functions (up to the choice of the correct representative) which have a weak derivative in $L^p$; in particular, $W^{1,p}$-functions have well-defined values at all points. This is a one-dimensional equivalent of the trace operator introduced in Section~\ref{sec:laplacian}.

\begin{definition}
Under our assumptions on $\Gamma$, for $1 \leq p \leq \infty$ we define
\begin{displaymath}
	W^{1,p}(\Gamma) = \{f \in C(\Gamma): f|_{e_k} \in W^{1,p}(e_k),\, k=1,\ldots,m\}.
\end{displaymath}
\end{definition}

\begin{theorem}
$W^{1,p}(\Gamma)$ is a Banach space for the canonical norm, which for $1\leq p < \infty$ is given by
\begin{displaymath}
	\|f\|_{W^{1,p}} = \left(\sum_{k=1}^m \int_{e_k} |f'(x)|^p + |f(x)|^p\,\textrm{d}x\right)^{1/p}.
\end{displaymath}
When $p=2$, $H^1(\Gamma) := W^{1,2}(\Gamma)$ is a Hilbert space with respect to the canonical inner product
\begin{displaymath}
	\langle f,g\rangle = \int_\Gamma f'\overline{g}' + f\overline{g} \,\textrm{d}x
	\equiv \sum_{k=1}^m \int_{e_k} f'(x)\overline{g'(x)} + f(x)\overline{g(x)}\,\textrm{d}x.
\end{displaymath}
\end{theorem}

\begin{remark}
Up to isometric isomorphism, the spaces $L^p(\Gamma)$, $C(\Gamma)$ and $W^{1,p}(\Gamma)$ are independent of the choice of local coordinates and relabelling of the edges.
\end{remark}

\begin{remark}[Dummy vertices]
\label{rem:dummy-vertices}
We call $v \in V$ a \emph{dummy vertex} if $\deg v=2$ and an \emph{essential vertex} otherwise. Suppose $\deg v = 2$, with $v \sim e_1,e_2$.
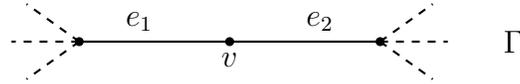
\begin{figure}[ht]
\centering
\begin{tikzpicture}
\draw[thick] (-2,0) -- (2,0);
\draw[thick,dashed] (-2,0) -- (-2.7,0.5);
\draw[thick,dashed] (-2,0) -- (-2.9,0);
\draw[thick,dashed] (-2,0) -- (-2.7,-0.5);
\draw[thick,dashed] (2,0) -- (2.7,0.5);
\draw[thick,dashed] (2,0) -- (2.9,0);
\draw[thick,dashed] (2,0) -- (2.7,-0.5);
\draw[fill] (-2,0) circle (1.5pt);
\draw[fill] (0,0) circle (1.5pt);
\draw[fill] (2,0) circle (1.5pt);
\node at (0,0) [anchor=north] {$v$};
\node at (-1.2,0) [anchor=south] {$e_1$};
\node at (1.2,0) [anchor=south] {$e_2$};
\node at (3.5,0) [anchor=west] {$\Gamma$};
\end{tikzpicture}
\end{figure}

Create a new graph $\widetilde\Gamma$ by replacing $e_1$ and $e_2$ with a single edge $e$ of length $|e_1|+|e_2|$, preserving all other incidence and adjacency relations. This ``deletes'' the dummy vertex $v$.
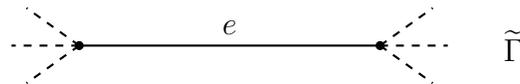
\begin{figure}[ht]
\centering
\begin{tikzpicture}
\draw[thick] (-2,0) -- (2,0);
\draw[thick,dashed] (-2,0) -- (-2.7,0.5);
\draw[thick,dashed] (-2,0) -- (-2.9,0);
\draw[thick,dashed] (-2,0) -- (-2.7,-0.5);
\draw[thick,dashed] (2,0) -- (2.7,0.5);
\draw[thick,dashed] (2,0) -- (2.9,0);
\draw[thick,dashed] (2,0) -- (2.7,-0.5);
\draw[fill] (-2,0) circle (1.5pt);
\draw[fill] (2,0) circle (1.5pt);
\node at (0,0) [anchor=south] {$e$};
\node at (3.5,0) [anchor=west] {$\widetilde\Gamma$};
\end{tikzpicture}
\end{figure}

This process is reversible: we may replace \emph{any} $x \in e$ with a dummy vertex, thus ``dividing'' $e$ into two edges. We speak of \emph{inserting} a dummy vertex at $x$.
\end{remark}

\begin{theorem}
There is an isometry between the graphs $\Gamma$ and $\widetilde\Gamma$, which induces an isometric isomorphism between the spaces $L^p(\Gamma)$ and $L^p(\widetilde\Gamma)$, $C(\Gamma)$ and $C(\widetilde\Gamma)$, and $W^{1,p}(\Gamma)$ and $W^{1,p}(\widetilde\Gamma)$, $1 \leq p \leq \infty$.
\end{theorem}

Inserting or deleting a dummy vertex does not ``change'' the graph at a metric or measure theoretic level. \emph{Any point} $x \in \Gamma$ may thus be treated as a vertex if it is convenient to do so.

\begin{remark}
It is easy to define higher order derivatives \emph{edgewise}: for example, if $f\in\nolinebreak W^{1,p}(\Gamma)$, and $e \in E$, we may write $f|_e'' \in L^p(e)$, and so $f|_e \in W^{2,p}(e)$, if there is a function $g \in L^p(e)$ (which we will call $f|_e''$) such that
\begin{displaymath}
	\int_e g \phi\,\textrm{d}x = -\int_e f' \phi'\,\textrm{d}x = \int_e f\phi''\,\textrm{d}x
\end{displaymath}
for all $\phi \in C_c^\infty (e)$, i.e. for all functions $\phi$ which are supported in the interior of the edge $e$. (Compare with Definition~\ref{def:delta-u}.)

However, there is no canonical way to define $W^{k,p}(\Gamma)$ (nor $C^k(\Gamma)$) for $k \geq 2$, since a (non-canonical) choice needs to be made regarding the conditions to impose on the derivatives of $f$ at the vertices; continuity of the derivatives is \emph{not} usually the most natural condition (as we will see). Sometimes it is useful to consider
\begin{equation}
\label{eq:wkp}
	\widetilde{W}^{k,p}(\Gamma) := \bigoplus_{k=1}^m W^{k,p}(e_k),
\end{equation}
which \emph{de facto} imposes \emph{no} conditions at the vertices.

Depending on what choice we make, the insertion or deletion of dummy vertices may affect the function spaces; for example, this is the case for $\widetilde{W}^{k,p}$ if $k \geq 1$.
\end{remark}

Notationally, if $f \in \widetilde{W}^{2,p}(\Gamma)$, so that $f'$ is continuous on each edge, and $e \sim v$, then we will write
\begin{equation}
\label{eq:vertex-derivative}
	\partial_\nu f|_{e} (v)
\end{equation}
for the (well-defined) value of the derivative of $f$ on the edge $e$ at its endpoint $v$, pointing \emph{into} $v$. This may be considered as an analogue of the outer normal derivative of a function $u: \overline{\Omega} \to \C$ on $\partial\Omega$.

\subsection{Laplacian and Schr{\"o}dinger operators on metric graphs}
\label{sec:laplacian-graph}

Given a compact metric graph $\Gamma = (V,E)$, how do we define operators like $-\Delta$ or $-\Delta+q$ on $\Gamma$? On each edge $e$, $-\Delta f$ should just be $-f|_e''$ (defined in the distributional sense), but what about the vertex conditions? Or, put differently, what should the operator domain be?

\emph{Form approach:} given a potential $q \in L^\infty (\Gamma,\R)$, define a form $a_q : H^1(\Gamma) \times H^1(\Gamma) \to \C$ by
\begin{equation}
\label{eq:form-graph}
	a_q (f,g) = \int_\Gamma f'(x)\overline{g'(x)} + q(x)f(x)\overline{g(x)}\,\textrm{d}x.
\end{equation}

\begin{lemma}
\label{lem:form-graph}
The form $a_q$ is a bounded Hermitian form which is $L^2$-elliptic on the space $H^1(\Gamma)\times H^1(\Gamma)$, that is, there exist $\xi \in \R$ and $\omega>0$ such that
\begin{displaymath}
	a_q (f,f) + \xi \|f\|_{L^2(\Gamma)}^2 \geq \omega \|f\|_{H^1(\Gamma)}^2
\end{displaymath}
for all $f \in H^1(\Gamma)$.
\end{lemma}

\begin{proof}
Exercise~2.2.
\end{proof}

\begin{proposition}
\label{prop:standard-laplacian}
The operator $A_q : D(A_q) \subset L^2(\Gamma) \to L^2(\Gamma)$ associated with $a_q$ is given by
\begin{equation}
\label{eq:operator-domain-graph}
	D(A_q) = \left\{f\in H^1(\Gamma): f|_e'' \in L^2(e)\,\, \forall e \in E, \text{ and } \sum_{e\sim v} \partial_\nu f|_e (v) = 0\,\,\forall v \in V\right\},
\end{equation}
and
\begin{displaymath}
	(A_q f)|_e = -f|_e'' + q|_ef|_e \qquad \text{in } L^2(e) \text{ for every edge } e \in E.
\end{displaymath}
\end{proposition}

\begin{proof}
Important exercise (Exercise~2.3).
\end{proof}

If, on every edge $e \in E$, $f|_e$ admits a weak second derivative $f|_e'' \in L^2(e)$ (in the sense of Definition~\ref{def:weak-derivative}, with $\Omega = e$), then we just write $f'' \in L^2(\Gamma)$ for the resulting edgewise defined function. Thus $A_q f(x) = -f''(x)+q(x)f(x)$ almost everywhere in $\Gamma$.

The characterisation \eqref{eq:operator-domain-graph} implies that at each vertex $v \in V$, two conditions are imposed on functions $f \in D(A_q)$:
\begin{enumerate}
\item $f$ should be \emph{continuous} at $v$ (this is contained in the condition $f \in H^1(\Gamma) \hookrightarrow C(\Gamma)$);
\item The sum of the (``outer normal'') derivatives of $f$ at all edges pointing into $v$ (cf.~\eqref{eq:vertex-derivative}), should be $0$.
\end{enumerate}
Since $D(A_q) \subset \widetilde{W}^{2,2}(\Gamma)$, these first-order derivatives are well defined by Lemma~\ref{lem:embedding-1d}. (2) is often called a \emph{Kirchhoff condition}, from the principle of ``current conservation'' (``what flows in, must flow out'').

Conditions (1) and (2) together are variously called standard, natural, Neumann--Kirchhoff, and continuity--Kirchhoff vertex conditions, and there are probably other names in use as well. (Observe that if $\deg v=1$, then they reduce to just a Neumann condition at the endpoint of the edge.) When $q=0$, $A_0$ is often called the \emph{Laplacian with standard vertex conditions}, a.k.a.\ the \emph{standard} or \emph{Kirchhoff Laplacian}.

\begin{remark}
\label{rem:other-vc}
There are of course many alternative conditions that can be imposed at the vertices, or a subset of the vertices. For example, imposing a Dirichlet condition at a vertex $v$ is easy; we simply require that $f(v)=0$ there instead (see Exercise 2.4). The ``standard'' conditions just introduced are the most common and are often considered ``natural'', since the vertices of a graph do not really play the same role as the boundary of a domain. Correspondingly, there will be no heat loss from a graph equipped with only standard conditions, as we will see below.

It can also be shown that inserting or deleting a dummy vertex (cf.\ Remark~\ref{rem:dummy-vertices}) equipped with standard conditions does not alter $A_q$ (up to unitary equivalence). This is not generally true of other vertex conditions.
\end{remark}

\begin{theorem}
\label{thm:spectral-theorem-graph}
Let $\Gamma$ be a connected, compact metric graph, and let $q \in L^\infty (\Gamma)$. The operator $A_q$ is self-adjoint and bounded from below on $L^2(\Gamma)$, and has compact resolvent. In particular, its spectrum is real and takes the form
\begin{displaymath}
	\lambda_1 < \lambda_2 \leq \lambda_3 \leq \ldots \to \infty,
\end{displaymath}
where each eigenvalue has finite algebraic $=$ geometric multiplicity. The associated eigenfunctions $\psi_k \sim \lambda_k$, $\psi_k \in D(A_q) \subset H^1(\Gamma)$, may be chosen real, and may be chosen to form an orthonormal basis of $L^2(\Gamma)$.
\end{theorem}

We will not deal with non-compact graphs here (that is, graphs with an infinite number of edges and/or infinite total length), where different types of spectrum may be present and quite a different set of issues and results emerges. These are explored comprehensively in \cite{dkmpt,kostenko,knm}.

\begin{remark}
The pair $(\Gamma,A_q)$, or, equivalently, the triple ``metric graph $\Gamma$ $+$ self-adjoint differential expression on the edges $+$ suitable (``self-adjoint'') vertex conditions'', is often called a \emph{quantum graph}. The standard book on the subject is \cite{berkolaiko}, although \cite{band17,berkolaiko16} may be more suitable as introductory sources for the layperson. It is most commonly held \cite{berkolaiko16} that the name arose as a shortening of the title of a 25-year-old article of Kottos and Smilansky, ``Quantum chaos on graphs'' \cite{kottos}, which studied (differential operators on) metric graphs as a model of \emph{quantum chaos}.

However, the study of differential operators on metric graphs goes back much further; there was a wave of activity in the 1980s, sometimes under the name $c^2$-networks, as in the seminal paper \cite{vonbelow}, which establishes a link between the spectrum of the standard Laplacian on an \emph{equilateral} compact metric graph (on which all edges have the same length) and the spectrum of a difference operator Laplacian on a corresponding discrete graph.

Even before that, what we now call quantum graphs were studied in the middle of last century with a view to applications in physics and chemistry, for example as a way to model nanoscale systems \cite{ruedenberg}.
\end{remark}

\subsection{First observations on the spectrum}
\label{sec:graphs-spectrum}

We will keep the assumptions of Theorem~\ref{thm:spectral-theorem-graph}. In light of that theorem, \emph{from now on we may, and will, assume that all functions are real valued}.

\begin{remark}
\label{rem:courant-fischer-graph}
The eigenvalues $\lambda_k$ may be characterised by the min-max and max-min principles of Theorem~\ref{thm:courant-fischer}; for example,
\begin{equation}
\label{eq:lambda1-rq-graph}
	\lambda_1 = \min_{0 \neq f \in H^1(\Gamma)} \frac{\int_\Gamma |f'(x)|^2+q(x)|f(x)|^2\,\textrm{d}x}{\int_\Gamma |f(x)|^2\,\textrm{d}x}.
\end{equation}
The quotient on the right-hand side of \eqref{eq:lambda1-rq-graph} is the Rayleigh quotient of $f$.
\end{remark}

\begin{remark}
Since $\Gamma$ is connected, it can be shown that $\lambda_1$ is simple and its eigenfunction $\psi_1$ may be chosen strictly positive everywhere in $\Gamma$ (Perron--Frobenius).
\end{remark}

\begin{remark}
\label{rem:standard-laplace-spectrum}
If $q=0$, then all eigenvalues of the standard Laplacian $A_0$ are nonnegative, since $a_0(f,f) \geq 0$ for all $f \in H^1(\Gamma)$. In this case, the standard Laplacian recalls the Neumann Laplacian; we recall that on a degree one vertex, the standard condition reduces to Neumann. In this case, instead of $\lambda_k$, we will write
\begin{displaymath}
	\mu_1 < \mu_2 \leq \mu_3 \leq \ldots \to \infty
\end{displaymath}
for the eigenvalues of $A_0$. This is partly because it will be convenient to distinguish these eigenvalues from other types of eigenvalues that will appear occasionally, and partly for historical reasons due to the analogy with the Neumann Laplacian. The spectrum of $A_0$ resembles the one of the Neumann Laplacian: we have $\mu_1 = 0$, the constant functions (which obviously satisfy both the continuity and the Kirchhoff conditions) being the eigenfunctions, and in this case
\begin{equation}
\label{eq:lambda2-rq-graph}
	0 < \mu_2 = \min_{\substack{0 \neq f \in H^1(\Gamma)\\ \int_\Gamma f\,\textrm{d}x = 0}}
	\frac{\int_\Gamma |f'|^2\,\textrm{d}x}{\int_\Gamma |f|^2\,\textrm{d}x}
\end{equation}
with equality if and only if $f$ is an eigenfunction for $\mu_2$, the condition $\int_\Gamma f\,\textrm{d}x = 0$ being the orthogonality condition, $\langle f,\psi_1\rangle = 0$.
\end{remark}

A notable difference between the eigenfunctions of graphs on the one hand, and eigenfunctions of domains or manifolds on the other, is that on graphs there is no global \emph{unique continuation principle}. Put differently, it is possible for an eigenfunction to be zero identically on an edge without being zero on the whole graph. A proper study of the eigenfunctions of the star graph of Example~\ref{ex:3-star} (or Example~\ref{ex:star} just below) yields an example; see also Exercise 3.1.

If $\Gamma$ consists of just a single edge (interval) $[0,\ell]$, and $q=0$, then the eigenvalues $\mu_k$ are the real numbers $\lambda$ for which the equation
\begin{displaymath}
	\cos (\sqrt{\lambda}\ell) = 0
\end{displaymath}
has a solution. A similar principle holds on all compact graphs: the eigenvalues are the solutions of a transcendental equation known as the \emph{secular equation}. However, even for simple graphs these equations generally become unmanageably complicated, and their solutions cannot generally be found explicitly:

\begin{center}
\noindent{\fbox{\parbox{280pt}{
For ``most'' quantum graphs it is generally impossible to obtain a closed analytic expression for the individual eigenvalues, even if we restrict to the Laplacian.
}}}
\end{center}

This is, however, far from being the whole story; this is still a far more explicit representation of the eigenvalues than one has on general domains and manifolds, where there is essentially no exact way to represent the eigenvalues at all; and there are, correspondingly, various quite powerful tools and approaches on quantum graphs that allow one to say a lot about the spectrum, and far more than using the corresponding approach on domains or manifolds.

This principle is a partial motivation for much of what we will do in Chapter~\ref{ch:geometric-spectral-theory}, which is one such approach. There are, however, several others, often very well developed, and although it would go beyond the scope of these notes to explore them here, it seems appropriate to mention a few of what are arguably the most important, and give a few references. This is however by no means a complete list, either in terms of the available techniques or in terms of the literature; for that, we refer to some of the many good books and surveys on the subject, including \cite{berkolaiko16,berkolaiko,kurasovbook}.

First, a systematic approach to the secular equation is offered by \emph{scattering matrices} and the \emph{secular determinant}, see, e.g. \cite[Section~1.4]{band17}, \cite[Section~5]{berkolaiko16}, or \cite[Section~2.1]{berkolaiko}.

Second, there are powerful, and often exact, \emph{trace formulae} available for the spectrum of the quantum graph as a whole, and typically based on the \emph{periodic orbits} (or, put differently, the set of closed paths on the graph), which results, for example, in exact expressions for the \emph{spectral measure}, or \emph{density of states}, of the graph in terms of its geometric and topological properties, or as an inverse formula for the Euler characteristic in terms of the spectrum. There is an advanced body of literature studying such formulae, which are often inspired by similar but generally not exact formulae on manifolds; in the context of quantum graphs these probably go back to the seminal paper \cite{kottos2}, see also, e.g. \cite{bhj} or, quite recently, \cite[Chapters 8 and 9]{kurasovbook}. 

Third, one can use \emph{M-functions} (known in the literature under a variety of different names, including Titchmarsh--Weyl M-functions and Dirichlet-to-Neumann operators) to study the spectrum in dependence on the vertex coupling conditions (cf.\ Exercise 2.4), which in turn can be used to analyse the spectrum of even just the standard Laplacian on a graph, see \cite[Chapters 17 and 18]{kurasovbook}.

\begin{example}[Secular equation of a star graph]
\label{ex:star}
Cf.\ \cite[Example~2.3]{berkolaiko16}. Consider the $3$-star graph of Example~\ref{ex:3-star}, but where now $e_k$ has length $\ell_k >0$, $k=1,2,3$, and suppose $q=0$. We will find the (or, more precisely, a) secular equation for the eigenvalues of $A_0$.

Choose local coordinates $e_k \simeq [0,\ell_k]$, such that $0\sim v_k$, $\ell_k \sim v_4$ (corresponding to the orientation in Figure~\ref{fig:3-star}-right). On each edge, in local coordinates, an eigenfunction $\psi$ for some eigenvalue $\lambda$ is just a solution of $-\psi''=\lambda\psi$ and thus given by
\begin{displaymath}
	\psi|_{e_k} (x) = A_k \cos (\sqrt{\lambda} x) + B_k \sin (\sqrt{\lambda} x), \qquad A_k,B_k \in \R,\,\, k=1,2,3.
\end{displaymath}
(Here we have used that all eigenvalues are nonnegative, as noted in Remark~\ref{rem:standard-laplace-spectrum}.)
For what values of $A_k,B_k,\lambda$ do these edgewise defined functions also satisfy the vertex conditions?

At $v_1,v_2,v_3$, the Kirchhoff condition reduces to $\psi'=0$: thus $B_k=0$, $k=1,2,3$.

Continuity at $v_4$ implies
\begin{equation}
\label{eq:3-star-continuity}
	A_1 \cos (\sqrt{\lambda}\ell_1) = A_2 \cos (\sqrt{\lambda} \ell_2) = A_3 \cos (\sqrt{\lambda} \ell_3).
\end{equation}
The Kirchhoff condition at $v_4$ implies
\begin{displaymath}
	-\sqrt{\lambda} A_1 \sin (\sqrt{\lambda} \ell_1) - \sqrt{\lambda} A_2 \sin (\sqrt{\lambda} \ell_2) - \sqrt{\lambda} A_3 \sin (\sqrt{\lambda}\ell_3) = 0.
\end{displaymath}
As long as $\lambda \neq 0$ (that is, excluding $\mu_1 = 0$, which is not of interest), this is equivalent to
\begin{equation}
\label{eq:3-star-kirchhoff}
	A_1 \sin (\sqrt{\lambda} \ell_1) + A_2 \sin (\sqrt{\lambda} \ell_2) + A_3 \sin (\sqrt{\lambda}\ell_3).
\end{equation}
As long as the common term in \eqref{eq:3-star-continuity} is nonzero, we can eliminated the dependence on the $A_k$ by dividing \eqref{eq:3-star-kirchhoff} by this common value, to arrive at the secular equation characterising the (nonzero) eigenvalues of $\Gamma$:
\begin{equation}
\label{eq:3-star-secular}
	\tan (\sqrt{\lambda} \ell_1) + \tan (\sqrt{\lambda} \ell_2) + \tan (\sqrt{\lambda} \ell_3) = 0.
\end{equation}
(The terms in \eqref{eq:3-star-continuity} can be zero only under a very particular assumption on the edge lengths of $\Gamma$. We leave it as an exercise to check that the secular equation may be written in the following alternative form, which is also valid in this case.)
\begin{equation}
\label{eq:3-star-secular-general}
\begin{aligned}
	&\sin(\sqrt{\lambda} \ell_1)\cos(\sqrt{\lambda} \ell_2) \cos(\sqrt{\lambda} \ell_3)+\\
	&\cos(\sqrt{\lambda} \ell_1)\sin(\sqrt{\lambda} \ell_2) \cos(\sqrt{\lambda} \ell_3)+\\
	&\cos(\sqrt{\lambda} \ell_1)\cos(\sqrt{\lambda} \ell_2) \sin(\sqrt{\lambda} \ell_3)=0.
\end{aligned}
\end{equation}
\end{example}

\section{Geometric spectral theory of quantum graphs}
\label{ch:geometric-spectral-theory}

\subsection{On domains}
\label{sec:gst-domains}

Given a bounded domain $\Omega \subset \R^d$, $d\geq 2$, with sufficiently smooth boundary, as in Section~\ref{sec:laplacian} we denote by
\begin{displaymath}
	0 < \lambda_1 < \lambda_2 \leq \lambda_3 \leq \ldots
\end{displaymath}
the eigenvalues of the Dirichlet Laplacian $-\Delta^D$ on $\Omega$, and by
\begin{displaymath}
	0 = \mu_1 < \mu_2 \leq \mu_3 \leq \ldots
\end{displaymath}
the eigenvalues of the Neumann Laplacian $-\Delta^N$ on $\Omega$.

How do the eigenvalues depend on $\Omega$? Write $\lambda_k = \lambda_k (\Omega)$, $\mu_k = \mu_k (\Omega)$.

\emph{Example 1:} The \emph{Weyl asymptotics} (a.k.a.\ \emph{Weyl's law}) asserts that
\begin{displaymath}
	\lambda_k(\Omega),\,\mu_k(\Omega) = C(|\Omega|)k^{2/d} + o(k^{2/d})
\end{displaymath}
as $k \to \infty$, for a constant $C(|\Omega|)>0$ depending only on the volume of $\Omega$ and the dimension $d$, whose precise value does not interest us here. In the Dirichlet case, this was originally proved by Hermann Weyl, a student of David Hilbert at G{\"o}ttingen, around 1911, in response to a conjecture formulated by leading physicists of the time, see \cite{arendt}.

\emph{Example 2:} Based on a conjecture of Lord Rayleigh in the late 19th Century, the following theorem was proved in the 1920s \cite{faber,krahn}.

\begin{theorem}[Faber--Krahn]
\label{thm:faber-krahn}
Let $B \subset \R^d$ be a ball of the same volume as $\Omega$, $|B|=|\Omega|$. Then
\begin{displaymath}
	\lambda_1 (\Omega) \geq \lambda_1 (B).
\end{displaymath}
There is equality iff $\Omega=B$ up to rigid transformations.
\end{theorem}

Since $\lambda_1$ governs the rate of decay of (the $L^2$-norms of) solutions to the heat equation (Section~\ref{sec:heat}), this theorem says that, among all objects of given volume, if there is perfect cooling at the boundary, the rate of heat loss (in this sense) is minimised when the object is a ball.

Theorem~\ref{thm:faber-krahn} is a starting point of geometric spectral theory, or more specifically \emph{shape optimisation} in spectral theory, see \cite{henrot06} or \cite{henrot17}, which also finds applications in nonlinear PDEs (where the eigenvalues often play the role of a critical or threshold value; an example of this happening on metric graphs is in \cite[Section~3]{cacciapuoti}) and related functional inequalities (for example, Theorem~\ref{thm:faber-krahn} together with the variational characterisation of the eigenvalues shows that $\lambda_1(B)$ is the smallest constant depending only on the volume of the domain for which Friedrichs' inequality is valid); these, in turn, are often needed in stability and domain perturbation analysis (e.g. \cite{daners,kennedyrobin}). Just like, or even more than, in the case of quantum graphs, only in extremely special cases is it possible to determine the eigenvalues exactly.

For the Neumann Laplacian, the baseline inequality is reversed.

\begin{theorem}[Szeg\H{o}--Weinberger]
\label{thm:szego-weinberger}
Under the same assumptions as Theorem~\ref{thm:faber-krahn},
\begin{displaymath}
	\mu_2 (\Omega) \leq \mu_2 (B),
\end{displaymath}
with equality iff $\Omega$ is a ball.
\end{theorem}

That is, in the case of perfect insulation, the ball exhibits the ``slowest $L^2$-convergence'' among all objects of the same volume to the equilibrium $=$ uniform distribution.

No corresponding upper bound is possible in Theorem~\ref{thm:faber-krahn}; e.g.\ in dimension two, just consider a sequence of increasingly long, thin rectangles. Likewise, no lower bound (other than zero) is possible in Theorem~\ref{thm:szego-weinberger}; so-called \emph{dumbbell domains} $\Omega_\varepsilon$, two fixed equal balls connected by a thin passage of width $\varepsilon$, satisfy $\mu_2 (\Omega_\varepsilon) \to 0$ as $\varepsilon \to 0$.

The key ingredient in the proof of Theorem~\ref{thm:faber-krahn} is the \emph{isoperimetric inequality}: if $U \subset \R^d$ is open and $B \subset \R^d$ is a ball such that
\begin{displaymath}
\begin{aligned}
	|\partial U| = |\partial B|, \qquad \text{then} \qquad |U| \leq |B|,\\
	|U| = |B|, \qquad \text{then} \qquad |\partial U| \geq |\partial B|
\end{aligned}
\end{displaymath}
(with equality iff $U=B$ up to rigid transformations and a negligible set).

The strategy involves \emph{symmetrisation} (or rearrangement): starting from $ \psi_1 \in H^1_0 (\Omega)$ a positive eigenfunction for $\lambda_1 (\Omega)$, using tools from geometric measure theory (the \emph{coarea formula} and the isoperimetric inequality applied to the level sets of $\psi_1$) we construct a new function $0 < \psi_1^\ast \in H^1_0(B)$, a symmetrisation of $\psi_1$ such that
\begin{displaymath}
	\|\psi_1^\ast\|_{L^2(B)} = \|\psi_1\|_{L^2(\Omega)} \qquad \text{but} \qquad \|\nabla \psi_1^\ast\|_{L^2(B)} \leq \|\nabla \psi_1\|_{L^2(\Omega)}.
\end{displaymath}
Based on the existence of such a function, Theorem~\ref{thm:faber-krahn} follows from the variational characterisation:
\begin{displaymath}
	\lambda_1 (\Omega) = \frac{\int_\Omega |\nabla \psi_1|^2\,\textrm{d}x}{\int_\Omega |\psi_1|^2\,\textrm{d}x}
	\geq \frac{\int_B |\nabla \psi_1^\ast|^2\,\textrm{d}x}{\int_B |\psi_1^\ast|^2\,\textrm{d}x}
	\geq \inf_{0 \neq u \in H^1_0(B)} \frac{\int_B |\nabla u|^2\,\textrm{d}x}{\int_B |u|^2\,\textrm{d}x} = \lambda_1 (B).
\end{displaymath}
A finer analysis, using the characterisation of equality in the isoperimetric inequality and the variational characterisation, shows that the inequality is an equality only if $\Omega$ is a ball.

\subsection{On graphs: Nicaise' inequality}
\label{sec:nicaise}

We will take the following standing assumptions: $\Gamma = (V,E)$ is a connected, compact metric graph; in particular, $V= \{v_1,\ldots, v_n\}$ and $E = \{e_1, \ldots, e_m\}$ are finite; each edge $e_k$ has finite length $|e_k|>0$, and the total length $|\Gamma|$ of $\Gamma$ is fixed,
\begin{displaymath}
	|\Gamma| = \sum_{k=1}^m |e_k| = L >0.
\end{displaymath}
We also assume $q=0$ and consider the standard Laplacian $A_0$. To emphasise the potential dependence of the spectrum on the graph, we will use the following notation for its eigenvalues:
\begin{displaymath}
	0 = \mu_1 (\Gamma) < \mu_2 (\Gamma) \leq \mu_3 (\Gamma) \leq \ldots \to \infty.
\end{displaymath}
We wish to understand how these eigenvalues are related to the geometry (and topology, and metric features) of $\Gamma$, akin to the examples sketched in Section~\ref{sec:gst-domains}. As discussed in Section~\ref{sec:graphs-spectrum}, in general it is essentially impossible to determine the $\mu_k$ explicitly (or even to find a convenient representation of the generally extremely complicated transcendental secular equation that characterises them).

Although the standard Laplacian is considered an analogue of the Neumann Laplacian, there is no version of Theorem~\ref{thm:szego-weinberger} for $\mu_2(\Gamma)$: there exist graphs $\Gamma_m$ such that $|\Gamma_m|=L$ for all $m \in \N$ but $\mu_2 (\Gamma_m) \to \infty$ as $m \to \infty$ (Exercise~3.1). Instead, the fundamental estimate for $\mu_2$ is a \emph{lower} bound reminiscent of the Faber--Krahn inequality for the Dirichlet Laplacian.

\begin{theorem}[Nicaise' inequality]
\label{thm:nicaise}
Let $I=[0,L]$ be an interval of length $L$. Then
\begin{displaymath}
	\mu_2 (\Gamma) \geq \mu_2 (I) = \frac{\pi^2}{L^2},
\end{displaymath}
with equality iff $\Gamma$ is a \emph{path graph}.
\end{theorem}

(A path graph is a graph consisting of a single edge upon removal of all dummy vertices, cf.\ Remark~\ref{rem:dummy-vertices}, and is thus equivalent to an interval.)

This theorem was originally discovered and proved by Nicaise in 1987 \cite{nicaise}, rediscovered by Friedlander in 2005 (with a different proof) \cite{friedlander}, and then rediscovered again by Kurasov and Naboko in 2014 \cite{kurasov}. The interpretation in terms of diffusion is that, if the graphs are perfectly insulated, so that no heat can escape, then, among all graphs of given total length, the interval exhibits the slowest convergence (in the usual $L^2$-sense) to equilibrium $=$ uniform heat distribution. Actually, there is a corresponding result for discrete graph Laplacians which is even older, going back to the 1970s \cite{fiedler}, see also \cite[Section~2]{bkkm1}.

Nicaise also proved a variant of this inequality in the case where \emph{at least one} vertex of $\Gamma$ is equipped with a Dirichlet vertex condition (see Exercise~2.4), and standard conditions at all others; any such Dirichlet vertex is a ``heat sink'' with perfect cooling to zero. In this case we will write
\begin{displaymath}
	0 < \lambda_1 (\Gamma) < \lambda_2 (\Gamma) \leq \ldots
\end{displaymath}
for the eigenvalues. (It is a short exercise to prove directly that the first eigenvalue is strictly positive, something which also follows immediately from the following theorem.)\footnote{We should, however, point out that the eigenvalues depend not just on $\Gamma$ as a \emph{metric graph} but also on exactly which vertex, or vertices, are being equipped with a Dirichlet vertex, a dependence which is implicitly suppressed in the notation we are using.}

\begin{theorem}[Nicaise, Dirichlet case]
\label{thm:nicaise-dirichlet}
Under our standing assumptions,
\begin{displaymath}
	\lambda_1 (\Gamma) \geq \lambda_1^{ND} (I) = \frac{\pi^2}{4L^2},
\end{displaymath}
the first eigenvalue of an interval of length $L$ with a Neumann condition at one endpoint and a Dirichlet condition at the other.
\end{theorem}

\textbf{Proof 1: Symmetrisation} (Friedlander's strategy)

Denote by $\psi_1$ an eigenfunction corresponding to $\lambda_1 (\Gamma)$; it may be shown that $\psi_1$ can be chosen strictly positive except at the Dirichlet vertices.\footnote{There is a minor technical assumption embedded in this assertion, which we will studiously ignore: we are assuming that $\Gamma$ is still connected \emph{after removal of all Dirichlet vertices}. If not, $\psi_1$ will generally be supported on one connected component of this punctured graph and be identically zero on all others.} Construct a \emph{symmetrised} (or \emph{rearranged}) function $\psi_1^\ast$ on $I=[0,L]$ as follows.

We first define the \emph{upper level sets} of $\psi_1$ by
\begin{displaymath}
	U_t := \{ x \in \Gamma : \psi_1(x) > t\}, \qquad t\geq 0;
\end{displaymath}
then $t \mapsto |U_t|$ (the total length of $U_t$) is a monotonically decreasing function from $L$ at $t=0$ to $0$ at $t = M:= \max_{x \in \Gamma} \psi_1(x)$. We will also denote by
\begin{displaymath}
	S_t := \{ x \in \Gamma : \psi_1 (x) = t\}, \qquad t \geq 0,
\end{displaymath}
the corresponding ``level surfaces'', $S_t = \partial U_t$, which in reality will generally be finite sets of points. (Indeed, since on each edge $\psi_1$ is a solution of $-\psi'' = \lambda \psi$ and thus a linear combination of sines and cosines, unless it is identically zero on an edge it only takes on any value a finite number of times. Thus each $S_t$ is indeed finite except possibly for $S_0$; what happens with $S_0$ depends on the minor technical assumption mentioned in the previous footnote.)

We define a function $\psi_1^\ast : [0,L] \to [0,M]$, the \emph{decreasing rearrangement} of $\psi_1$, by the rule
\begin{displaymath}
	\psi_1^\ast (x) := t \qquad \text{iff} \qquad x = |U_t|, \, x \in [0,L].
\end{displaymath}
It is defined in such a way that its upper level sets have the same total length as the upper level sets of $\psi_1$:
\begin{equation}
\label{eq:cavalieri}
	|U_t^\ast| = |\{y \in [0,L]: \psi_1^\ast (y) > t\}| = x = |U_t| = |\{y \in \Gamma: \psi_1(y) > t\}|.
\end{equation}
Thus $\psi_1^\ast$ is monotonically decreasing, with $\psi_1^\ast (0)=M$ and $\psi_1^\ast (L) = 0$. See Figure~\ref{fig:rearrangement}.

\begin{figure}[ht]
\centering
\begin{tikzpicture}[scale=1.3]
\draw[thick,->] (-0.75,0) -- (4.25,0);
\draw[thick,->] (0,-0.5) -- (0,2.5);
\draw[thick] (0,1.7) .. controls (1.5,0.8) .. (2.5,0.7) .. controls (2.7,0.6) .. (3,0);
\node at (0,1.7) [anchor=east] {$M$};
\node at (0,0) [anchor=north east] {$0$};
\node at (3,0) [anchor=north] {$L$};
\node at (2.7,0.6) [anchor=south west] {$\psi_1^\ast$};
\node at (1.5,0) [anchor=north] {$x = |U_t|$};
\node at (0,0.865) [anchor=east] {$t$};
\draw[thick,dotted] (1.5,0) -- (1.5,0.865) -- (0,0.865);
\end{tikzpicture}
\caption{The rearranged function $\psi_1^\ast$ on $[0,L]$.}
\label{fig:rearrangement}
\end{figure}
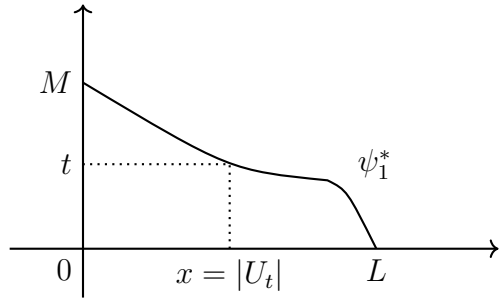

A technical lemma (which we will omit) shows that $\psi_1^\ast \in H^1(0,L)$, while Cavalieri's principle (using \eqref{eq:cavalieri}) implies that
\begin{displaymath}
	\|\psi_1^\ast\|_{L^p([0,L])} = \|\psi_1\|_{L^p(\Gamma)}
\end{displaymath}
for all $p \in [1,\infty)$, and in particular for $p=2$.

\begin{lemma}
We have
\begin{displaymath}
	\int_\Gamma |\psi_1'|^2\,\textrm{d}x \geq \int_0^L |(\psi_1^\ast)'|^2\,\textrm{d}x.
\end{displaymath}
\end{lemma}

\begin{proof}
The proof, which can comfortably be omitted on a first reading, relies on the coarea formula, a kind of a change of variable formula, which in dimension one is as follows: for any suitable function $\varphi$ (piecewise continuous and integrable is sufficient),
\begin{displaymath}
	\int_\Gamma \varphi(x) |\psi_1'(x)|\,\textrm{d}x = \int_0^M \left(\sum_{x \in S_t} \varphi(x)\right)\,\textrm{d}t,
\end{displaymath}
part of the assertion being that the integrand on the right-hand side is in fact integrable. Before we proceed, we also make two observations: firstly,
\begin{displaymath}
	\sum_{x \in S_t} |\psi_1'(x)| \geq \left(\sum_{x \in S_t} \frac{1}{|\psi_1'(x)|}\right)^{-1}
\end{displaymath}
almost everywhere (also noting that $\psi_1'(x)$ can vanish or be undefined only on a null set), since
\begin{displaymath}
	\sum_{x \in S_t} \frac{1}{|\psi_1'(x)|} \sum_{y \in S_t} |\psi_1'(y)| 
	\geq\sum_{x \in S_t} 1 = \# S_t \geq 1.
\end{displaymath}
secondly, if $y_t \in [0,L]$ is the unique point such that $\psi_1^\ast (y_t) = t$, i.e. $\{y_t\} = S_t^\ast = \partial U_t^\ast$, then the relation $|U_t| = |U_t^\ast|$ implies
\begin{displaymath}
	\sum_{x \in S_t} \frac{1}{|\psi_1'(x)|} = \sum_{y \in S_t^\ast} \frac{1}{|(\psi_1^\ast)'(y)|} = \frac{1}{|(\psi_1^\ast)'(y_t)|}
\end{displaymath}
for almost all $t \in [0,M]$. With these two observations, we can give the main calculation:
\begin{displaymath}
\begin{aligned}
	\int_\Gamma |\psi_1'(x)|^2\,\textrm{d}x
	&= \int_0^M \sum_{x \in S_t} |\psi_1'(x)|\,\textrm{d}t\\
	&\geq \int_0^M \frac{1}{\sum_{x \in S_t}\frac{1}{|\psi_1'(x)|}}\,\textrm{d}t\\
	&= \int_0^M \frac{1}{1/|(\psi_1^\ast)'(y_t)|}\,\textrm{d}t\\
	&= \int_0^M |(\psi_1^\ast)'(y_t)|\,\textrm{d}t = \int_0^L |(\psi_1^\ast)'(y)|^2\,\textrm{d}y,
\end{aligned}
\end{displaymath}
where the first line is the coarea formula, the second follows from the first observation, the third from the second observation, and the final line is another application of the coarea formula.
\end{proof}

The proof of Theorem~\ref{thm:nicaise-dirichlet} now follows the same reasoning as the theorem of Faber--Krahn in $\R^d$: since $\psi_1^\ast (L)=0$,
\begin{displaymath}
	\lambda_1 (\Gamma) = \frac{\int_\Gamma |\psi_1'|^2\,\textrm{d}x}{\int_\Gamma |\psi_1|^2\,\textrm{d}x}
	\geq \frac{\int_0^L |(\psi_1^\ast)'|^2\,\textrm{d}x}{\int_0^L |\psi_1^\ast|^2\,\textrm{d}x}
	\geq \inf_{\substack{u \in H^1(0,L)\\ u(L)=0}} \frac{\int_0^L |u'|^2\,\textrm{d}x}{\int_0^L |u|^2\,\textrm{d}x} = \lambda_1^{ND}(0,L)
	=\frac{\pi^2}{4L^2}.
\end{displaymath}
To prove Theorem~\ref{thm:nicaise} using this approach, we apply Theorem~\ref{thm:nicaise-dirichlet} to the \emph{nodal domains} of the second eigenfunctions.

\begin{definition}
Let $\varphi \in C(\Gamma)$. We call
\begin{enumerate}
\item[\textbf{(a)}] the closed set $N(\varphi) := \{x \in \Gamma: \varphi(x)=0\}$ the \emph{nodal set} of $\varphi$;
\item[\textbf{(b)}] $\Omega \subset \Gamma$ a \emph{nodal domain} if it is the closure of a connected component of $\Gamma \setminus N(\varphi)$.
\end{enumerate}
\end{definition}

(The decision to take $\Omega$ as \emph{closed} is somewhat non-standard and made for technical reasons. In general, on domains and manifolds, the nodal domains of a function $\varphi$ are usually the (open) connected components of the complement of the nodal set.) Note that $\varphi(x) = 0$ for all $x \in \partial\Omega \subset \Gamma$, and $\varphi$ does not change sign on $\Omega$.

\begin{example}
If $\Gamma = I = [0,L]$ is just a bounded interval, and $\psi_k \sim \lambda_k$ is a $k$th eigenfunction of $A_q$ on $I$ ($q \in L^\infty(I)$), then by Sturm--Liouville theory, $\psi_k$ has exactly $k-1$ zeros in the interior of $I$, and thus exactly $k$ nodal domains $\Omega_1,\ldots,\Omega_k$.

As a more explicit example, if $q=0$ and $k=2$, then $\psi_2(x) = A\cos (\frac{\pi x}{L})$ is an eigenfunction for $\lambda_2 = \frac{\pi^2}{L^2}$; $N(\psi_2) = \{\frac{L}{2}\}$ (independently of $A \neq 0$), and so $\Omega_1 = [0,\frac{L}{2}]$, $\Omega_2 = [\frac{L}{2},L]$.
\end{example}

\begin{remark}
If $k \geq 2$, then $\psi_k \sim \mu_k(\Gamma)$ changes sign on $\Gamma$ (since it is orthogonal in $L^2(\Gamma)$ to the function $\psi_1$, which is positive everywhere). In particular, if $k=2$, then there exist nodal domains $\Omega^+$ and $\Omega^-$ where $\psi_k$ is positive and negative, respectively.
\end{remark}

This principle applies far more generally than just on graphs. The analysis of the number of nodal domains of the eigenfunctions as a function of $k$ (be it on a Euclidean domain, a manifold, or a graph $\Gamma$) is a longstanding area of investigation:
\begin{itemize}
\item Sturm: in 1D, $\psi_k$ has \emph{exactly} $k$ nodal domains.
\item Courant: in $\Omega \subset \R^d$, $\psi_k$ has $\leq k$ nodal domains.
\item Pleijel: in $\Omega \subset \R^d$,
\begin{displaymath}
	\limsup_{k \to \infty} \frac{\# \text{nodal domains of }\psi_k}{k} < 1
\end{displaymath}
(this uses Faber--Krahn, Theorem~\ref{thm:faber-krahn}), see \cite{pleijel}.
\item On graphs: Alon, Band, Berkolaiko, Smilansky among others: for ``most'' graphs (more precisely, those for which the eigenfunctions never vanish on any essential vertex),
\begin{displaymath}
	k - (\#E - \#V + 1) \leq \#\text{nodal domains of }\psi_k \leq k
\end{displaymath}
(the number $\#E - \#V + 1$ is known as the (first) \emph{Betti number} of $\Gamma$, and is equal to the number of independent cycles in $\Gamma$). This is a topic of ongoing study; see, for example, \cite{alon,berkolaiko08}. The situation is more complicated if one or more $\psi_k$ vanish at one or more essential vertices, see Exercise~3.1 and \cite{hofmann21}.
\end{itemize}

\begin{lemma}
\label{lem:nodal}
Suppose $\psi_k$ is an eigenfunction for $\mu_k(\Gamma)$ on $\Gamma$ and $\Omega \subset \Gamma$ is any nodal domain of $\psi_k$. Denote by $\lambda_1 (\Omega)$ the first eigenvalue of the Laplacian on $\Omega$ with Dirichlet conditions on $\partial\Omega$ and standard conditions at all other vertices (see Exercise~2.4), i.e.
\begin{displaymath}
	\lambda_1 (\Omega) = \inf_{\substack{0 \neq u \in H^1(\Omega)\\ u(x)=0\,\forall x \in \partial\Omega}}
	\frac{\int_\Omega |u'|^2+q|u|^2\,\textrm{d}x}{\int_\Omega |u|^2\,\textrm{d}x}.
\end{displaymath}
Then $\mu_k (\Gamma) = \lambda_1 (\Omega)$.
\end{lemma}

\begin{proof}
%
By assumption $\psi_k$ satisfies the eigenvalue equation $-\psi_k = \mu_k \psi_k$ strongly (indeed pointwise) in $\Gamma$, and hence in $\Omega$. But it also satisfies all vertex conditions in $\Omega$ by construction. Thus, it is equal to \emph{some} eigenfunction on $\Omega$; in particular, $\mu_k = \lambda_j (\Omega)$ for some $k\geq 1$.

Now since $\psi_k$ does not change sign in $\Omega$, and $\lambda_1(\Omega)$ is the only eigenvalue on $\Omega$ with a non-sign-changing eigenfunction, we must have $j=1$.
\end{proof}

\begin{proof}[First proof of Theorem~\ref{thm:nicaise}]
Fix an eigenfunction $\psi_2 \sim \mu_2 (\Gamma)$, then $\psi_2$ has (at least) two nodal domains $\Omega^+$, $\Omega^-$; at least one of them, say $\Omega^+$, has total length $\leq L/2$. Since $\partial\Omega^+ \neq \emptyset$, i.e. $\Omega^+$ has at least one Dirichlet vertex, by Lemma~\ref{lem:nodal} and Theorem~\ref{thm:nicaise-dirichlet},
\begin{displaymath}
	\mu_2 (\Gamma) = \lambda_1 (\Omega^+) \geq \frac{\pi^2}{4|\Omega^+|^2} \geq \frac{\pi^2}{L^2}.
\end{displaymath}
\end{proof}

The symmetrisation method allows a generalisation of the inequality to the higher eigenvalues:

\begin{theorem}[Friedlander]
For all $k \geq 1$,
\begin{displaymath}
	\mu_k (\Gamma) \geq \frac{\pi^2 (k-1)^2}{L^2}.
\end{displaymath}
The minimising graph is an \emph{equilateral $k$-star}.
\end{theorem}

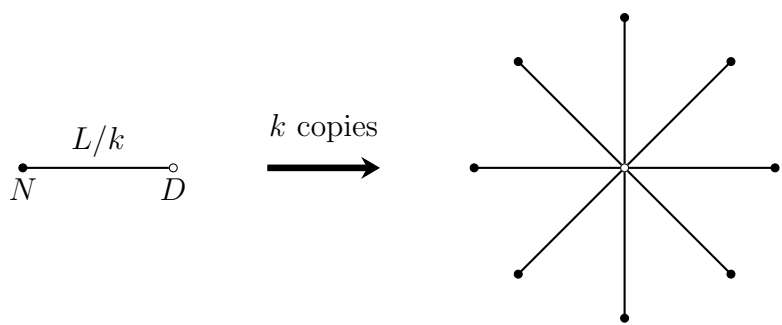
\begin{figure}[H]
\centering
\begin{tikzpicture}
\draw[thick] (-4,0) -- (-2,0);
\draw[fill] (-4,0) circle (1.5pt);
\draw[fill,white] (-2,0) circle (1.5pt);
\draw (-2,0) circle (1.5pt);
\node at (-4,0) [anchor=north] {$N$};
\node at (-2,0) [anchor=north] {$D$};
\node at (-3,0) [anchor=south] {$L/k$};
\draw[-{Stealth[scale=0.5,angle'=60]},line width=2.5pt] (-0.75,0) -- (0.75,0);
\node at (0,0.2) [anchor=south] {$k$ copies};
\draw[thick] (2,0) -- (6,0);
\draw[thick] (4,2) -- (4,-2);
\draw[thick] (5.414,1.414) -- (2.586,-1.414);
\draw[thick] (2.586,1.414) -- (5.414,-1.414);
\draw[fill,white] (4,0) circle (1.5pt);
\draw (4,0) circle (1.5pt);
\draw[fill] (2,0) circle (1.5pt);
\draw[fill] (6,0) circle (1.5pt);
\draw[fill] (4,2) circle (1.5pt);
\draw[fill] (4,-2) circle (1.5pt);
\draw[fill] (5.414,1.414) circle (1.5pt);
\draw[fill] (2.586,-1.414) circle (1.5pt);
\draw[fill] (2.586,1.414) circle (1.5pt);
\draw[fill] (5.414,-1.414) circle (1.5pt);
\end{tikzpicture}
\end{figure}

\begin{proof}[Rough idea of proof]
If $\psi_k$ has $k$ nodal domains, then at least one has length $\leq L/k$; apply Theorem~\ref{thm:nicaise-dirichlet}. Otherwise, there exists a linear combination of $\psi_1,\ldots,\psi_k$ which does have $k$ nodal domains. Apply the symmetrisation technique to one of these.
\end{proof}

\begin{remark}
We note in passing the high multiplicity of the minimising graph (cf.\ Exercise~3.1): $0 = \mu_1 (\Gamma) < \mu_2 (\Gamma) = \ldots = \mu_k (\Gamma) < \mu_{k+1} (\Gamma)$ for the equilateral $k$-star $\Gamma$. This is to be compared with the numerically observed property of domains in $\R^2$ which minimise $\lambda_k (\Omega)$ \cite{antunes}: there, numerical evidence strongly suggests that the multiplicity of $\lambda_k$ for the optimal domain grows with $k$, but there is no known proof.

On the other hand, the values in Friedlander's theorem diverge sharply from the values of $\mu_k (\Gamma)$ given by Weyl's law for any \emph{fixed} graph $\Gamma$. This is starkly at odds with \emph{P\'olya's conjecture} (see, e.g. \cite{colbois}) for the Dirichlet and Neumann Laplacian eigenvalues, which posits that the first term in the Weyl asymptotics is always a lower bound for the Dirichlet eigenvalues and an upper bound for the Neumann eigenvalues, on any domain. It seems likely that this is related to the failure of the unique continuation principle mentioned in Section~\ref{sec:graphs-spectrum} and to the nodal domain counts explored in \cite{hofmann21}.
\end{remark}

\subsection{Surgery}
\label{sec:surgery}

\textbf{Second proof of Theorem~\ref{thm:nicaise}: the ``doubling trick''} (Nicaise/Kurasov--Naboko)

The original proof given by Nicaise was based on the idea of finding a \emph{double covering} of the graph $\Gamma$. The following version of the proof, given by Kurasov and Naboko \cite{kurasov}, is more explicit and elegantly uses a classical theorem of Euler.

Given $\Gamma$, construct a new, ``doubled'' graph $\Gamma_2$ by replacing each edge $e$ of $\Gamma$ with two identical parallel edges, each of the same length as $e$. Then $|\Gamma_2| = 2L$.

\noindent Duplicate the eigenfunction $\psi_2 \in H^1 (\Gamma)$ on each edge to create a new function $\tilde \psi_2 \in\nolinebreak H^1(\Gamma_2)$; then by construction
\begin{displaymath}
\begin{aligned}
	\int_{\Gamma_2} |\tilde\psi_2'|^2\,\textrm{d}x &= 2\int_\Gamma |\psi_2'|^2\,\textrm{d}x\\
	\int_{\Gamma_2} |\tilde\psi_2|^2\,\textrm{d}x &= 2\int_\Gamma |\psi_2|^2\,\textrm{d}x\\
	\int_{\Gamma_2} \tilde\psi_2\,\textrm{d}x &= 2\int_\Gamma \psi_2\,\textrm{d}x = 0,
\end{aligned}
\end{displaymath}
and so
\begin{displaymath}
	\mu_2 (\Gamma) = \frac{\int_\Gamma |\psi_2'|^2\,\textrm{d}x}{\int_\Gamma |\psi_2|^2\,\textrm{d}x}
	= \frac{\int_{\Gamma_2} |\tilde\psi_2'|^2\,\textrm{d}x}{\int_{\Gamma_2} |\tilde\psi_2|^2\,\textrm{d}x}
	\geq \inf_{\substack{0 \neq u \in H^1(\Gamma_2)\\ \int_{\Gamma_2} u\,\textrm{d}x=0}} 
	\frac{\int_{\Gamma_2} |u'|^2\,\textrm{d}x}{\int_{\Gamma_2} |u|^2\,\textrm{d}x} = \mu_2 (\Gamma_2).
\end{displaymath}
(A similar argument shows that $\mu_k (\Gamma) \geq \mu_k (\Gamma_2)$ for all $k \geq 1$.)

\begin{figure}[ht]
\centering
\begin{tikzpicture}
\draw[thick] (-6.414,1.414) -- (-5,0);
\draw[thick] (-3.586,1.414) -- (-5,0);
\draw[thick] (-5,0) -- (-5,-2);
\draw[fill] (-6.414,1.414) circle (1.5pt);
\draw[fill] (-3.586,1.414) circle (1.5pt);
\draw[fill] (-5,0) circle (1.5pt);
\draw[fill] (-5,-2) circle (1.5pt);
\node at (-5.5,0) [anchor=east] {$\Gamma$};
\draw[thick] (-1.414,1.414) to[bend left=30] (0,0);
\draw[thick] (-1.414,1.414) to[bend right=30] (0,0);
\draw[thick] (1.414,1.414) to[bend left=30] (0,0);
\draw[thick] (1.414,1.414) to[bend right=30] (0,0);
\draw[thick] (0,0) to[bend left=30] (0,-2);
\draw[thick] (0,0) to[bend right=30] (0,-2);
\draw[fill] (-1.414,1.414) circle (1.5pt);
\draw[fill] (1.414,1.414) circle (1.5pt);
\draw[fill] (0,0) circle (1.5pt);
\draw[fill] (0,-2) circle (1.5pt);
\node at (0.75,-0.25) [anchor=west] {$\Gamma_2$};
\draw[thick,->] (3.586,1.414) to[bend left=30] (5,0.3);
\draw[thick,<-] (3.586,1.414) to[bend right=30] (4.8,-0.1);
\draw[thick,->] (6.414,1.414) to[bend left=30] (5.2,-0.1);
\draw[thick,<-] (6.414,1.414) to[bend right=30] (5,0.3);
\draw[thick,->] (5.2,-0.1) to[bend left=30] (5,-2);
\draw[thick,<-] (4.8,-0.1) to[bend right=30] (5,-2);
\draw[fill] (-3.586,1.414) circle (1pt);
\draw[fill] (6.414,1.414) circle (1pt);
\draw[fill] (5.2,-0.1) circle (1pt);
\draw[fill] (4.8,-0.1) circle (1pt);
\draw[fill] (5,0.3) circle (1pt);
\draw[fill] (5,-2) circle (1pt);
\node at (5.75,-0.25) [anchor=west] {$C$};
\end{tikzpicture}
\caption{The original graph $\Gamma$ (left); the ``doubled graph'' $\Gamma_2$ (centre); an Eulerian cycle $C$ (right), which forms a closed cycle in $\Gamma_2$, traversing every edge exactly once.}
\end{figure}
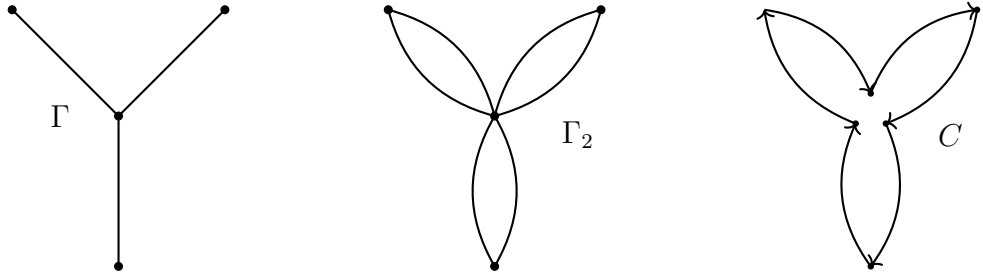

\begin{theorem}[Euler, 1736, ``Bridges of K{\"o}nigsberg'']
If all vertices of a graph $\Gamma_2$ have even degree, then there exists an \emph{Eulerian cycle} $C$ in $\Gamma_2$ traversing each edge of $\Gamma_2$ exactly once.
\end{theorem}

(The cycle may pass through one or more vertices more than once.) Map each function $u \in H^1 (\Gamma_2)$ onto a function $\tilde u \in H^1(C)$ in the obvious way; then all integrals are preserved. Since not every function in $H^1(C)$ can be transformed back into an $H^1(\Gamma_2)$-function, we may identify $H^1(\Gamma_2)$ with a (proper) subspace of $H^1(C)$, and so
\begin{align*}
	\mu_2(\Gamma) \geq \mu_2(\Gamma_2) &= \inf_{\substack{0 \neq u \in H^1(\Gamma_2)\\ \int_{\Gamma_2} u\,\textrm{d}x=0}} 
	\frac{\int_{\Gamma_2} |u'|^2\,\textrm{d}x}{\int_{\Gamma_2} |u|^2\,\textrm{d}x} = \mu_2 (\Gamma_2) \\
 &\geq
	\inf_{\substack{0 \neq u \in H^1(C) \\ \int_{C} u\,\textrm{d}x=0}} 
	\frac{\int_{C} |u'|^2\,\textrm{d}x}{\int_{C} |u|^2\,\textrm{d}x} = \mu_2 (C) = \mu_2 (C).
\end{align*}
Now $C$ is a cycle (circle) of length $2L$, which corresponds to an interval of length $2L$ with periodic boundary conditions and first nonzero eigenvalue $\mu_2 (C) = \frac{4\pi^2}{(2L)^2}$. Summarising,
\begin{displaymath}
	\mu_2 (\Gamma) \geq \mu_2 (\Gamma_2) \geq \mu_2 (C) = \frac{4\pi^2}{(2L)^2} = \frac{\pi^2}{L^2}.
\end{displaymath}
This completes the alternative proof of Theorem~\ref{thm:nicaise}.

\begin{remark}
\label{rem:eulerian-cycle}
If $\Gamma$ already contains an Eulerian cycle, then the ``doubling trick'' is not necessary:
\begin{displaymath}
	\mu_2 (\Gamma) \geq \mu_2 (\text{cycle of length } L) = \frac{4\pi^2}{L^2}.
\end{displaymath}
Actually, this inequality holds for a more general class of graphs, known as \emph{doubly connected graphs} (Theorem of Band--L\'evy, \cite[Theorem~2.1(2)]{bandlevy}). Intuitively, graphs which are more connected, i.e. have more non-overlapping paths between any pairs of points, necessarily have larger $\mu_2$ and thus faster convergence of diffusion processes to equilibrium. Higher graph connectivity and its effect on the eigenvalues is also explored in \cite{bkkm1}. This closely parallels results for discrete graph Laplacians which have been known for many decades \cite{fiedler,mohar}, where the discrete counterpart of $\mu_2$ is even called the \emph{algebraic connectivity} of the graph.
\end{remark}

We shall explore the related principle that loosening connections, or ``cutting through the graph'' lowers the eigenvalues. This is a prototypical \emph{surgery principle}: one examines how making a local topological, geometric or metric change to a graph (``surgery'') affects its Laplacian spectrum. Some such principles were implicit in the works of Nicaise \cite{nicaise} and Friedlander \cite{friedlander}, but only started being studied systematically in the 2010s with \cite{kurasov,kmn}, and are arguably now recognised as a powerful collection of techniques. The standard reference is probably \cite{bkkm2}.

\begin{definition}
We say $\widetilde\Gamma$ is formed from $\Gamma$ by \emph{cutting through the vertex $v \in V(\Gamma)$} if $v$ is replaced by $p \geq 2$ vertices $v_1,\ldots,v_p \in V(\widetilde\Gamma)$ such that all other incidence and adjacency relations are preserved.
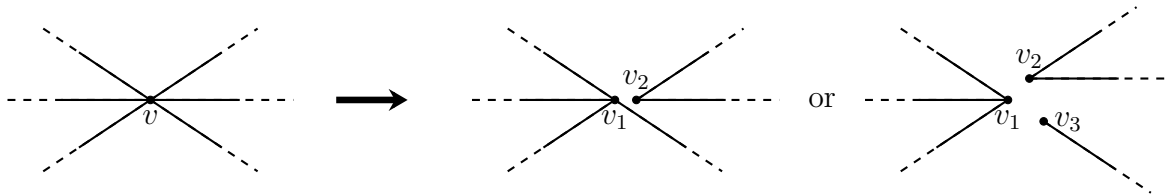
\begin{figure}[ht]
\centering
\begin{tikzpicture}[scale=0.95]
\draw[thick,dashed] (-5,1) -- (-2,-1);
\draw[thick,dashed] (-5.5,0) -- (-1.5,0);
\draw[thick,dashed] (-5,-1) -- (-2,1);
\draw[thick] (-4.5,0.6667) -- (-2.5,-0.6667);
\draw[thick] (-4.75,0) -- (-2.25,0);
\draw[thick] (-4.5,-0.6667) -- (-2.5,0.6667);
\draw[fill] (-3.5,0) circle (1.5pt);
\node at (-3.5,0) [anchor=north] {$v$};
\draw[-{Stealth[scale=0.5,angle'=60]},line width=2.5pt] (-0.9,0) -- (0.1,0);
\draw[thick,dashed] (1,0) -- (3,0);
\draw[thick] (1.75,0) -- (3,0);
\draw[thick,dashed] (1.5,1) -- (4.5,-1);
\draw[thick] (2,0.6667) -- (4,-0.6667);
\draw[thick,dashed] (1.5,-1) -- (3,0);
\draw[thick] (2,-0.6667) -- (3,0);
\draw[fill] (3,0) circle (1.5pt);
\node at (3,0) [anchor=north] {$v_1$};
\draw[thick,dashed] (3.3,0) -- (5.3,0);
\draw[thick] (3.3,0) -- (4.55,0);
\draw[thick,dashed] (3.3,0) -- (4.8,1);
\draw[thick] (3.3,0) -- (4.3,0.6667);
\draw[fill] (3.3,0) circle (1.5pt);
\node at (3.3,0) [anchor=south] {$v_2$};
\node at (5.9,0) {or};
\draw[thick,dashed] (6.5,0) -- (8.5,0);
\draw[thick] (7.25,0) -- (8.5,0);
\draw[thick,dashed] (7,1) -- (8.5,0);
\draw[thick,dashed] (7,-1) -- (8.5,0);
\draw[thick] (7.5,0.6667) -- (8.5,0);
\draw[thick] (7.5,-0.6667) -- (8.5,0);
\draw[fill] (8.5,0) circle (1.5pt);
\node at (8.5,0) [anchor=north] {$v_1$};
\draw[thick,dashed] (9,-0.3) -- (10.5,-1.3);
\draw[thick] (9,-0.3) -- (10,-0.9667);
\draw[fill] (9,-0.3) circle (1.5pt);
\node at (9,-0.3) [anchor=west] {$v_3$};
\draw[thick,dashed] (8.8,0.3) -- (10.8,0.3);
\draw[thick] (8.8,0.3) -- (10.05,0.3);
\draw[thick,dashed] (8.8,0.3) -- (10.3,1.3);
\draw[thick] (8.8,0.3) -- (9.8,0.9667);
\draw[fill] (8.8,0.3) circle (1.5pt);
\node at (8.8,0.3) [anchor=south] {$v_2$};
\end{tikzpicture}
\caption{Two examples of cutting through a vertex of degree $6$. The first (centre) has rank $1$; the second (right) has rank $2$, and is in fact a rank $1$ cut of the first.}
\end{figure}

The \emph{rank} of the cut is the number of new vertices created, i.e.\ $p-1$. The inverse process is called \emph{gluing} vertices.
\end{definition}

\begin{lemma}
\label{lem:cut}
Suppose $\widetilde\Gamma$ is formed from $\Gamma$ by making a rank $r\geq 1$ cut of a vertex $v$ of $\Gamma$, preserving standard conditions at all vertices. Then
\begin{equation}
\label{eq:cut}
	\mu_k (\Gamma) \geq \mu_k (\widetilde\Gamma) \geq \mu_{k-r} (\Gamma)
\end{equation}
for all $k \geq 1$ (where for the lower bound we assume $k \geq r+1$).
\end{lemma}

See also Exercise 3.2 or \cite[Section~3]{bkkm2}. The process of cutting/gluing is more complicated if other vertex conditions are imposed, since then, unlike with the case of standard conditions, there may be no natural conditions the newly created vertices should satisfy.

\begin{proof}
A rank $r$ cut removes $r$ continuity conditions at $v$: we may identify $C(\Gamma)$ with a subspace of $C(\widetilde\Gamma)$ of codimension $r$. It follows from the definition of the $H^1$-spaces that $H^1(\Gamma) \hookrightarrow C(\Gamma)$ may also be identified with a codimension $r$ subspace of $H^1(\widetilde\Gamma)$. (Cf.\ also Exercise~2.1(d).)

\eqref{eq:cut} now follows from the min-max principle (Theorem~\ref{thm:courant-fischer}/Remark~\ref{rem:courant-fischer-graph}).
\end{proof}

\begin{corollary}[Weyl asymptotics]
\label{cor:weyl-graph}
Let $\Gamma$ be a compact metric graph with total length $L$. Then
\begin{displaymath}
	\mu_k (\Gamma) = \frac{\pi^2 k^2}{L^2} + O(k)
\end{displaymath}
as $k \to \infty$.
\end{corollary}

\begin{proof}
Exercise 3.3.
\end{proof}

How can we make $\mu_2$ as large as possible using this idea? We saw (Exercise~3.1) that there cannot be a complementary upper bound to Theorem~\ref{thm:nicaise}. But we can still exploit the principle that gluing vertices increases the eigenvalues:

\begin{theorem}
\label{thm:flower}
(Cf.\ \cite[Theorem~4.2]{kkmm}.) Let $\Gamma$ be a compact metric graph with total length $L>0$ and $m$ edges. Then
\begin{displaymath}
	\mu_2 (\Gamma) \leq \frac{\pi^2m^2}{L^2} = \frac{\pi^2}{(\text{mean edge length})^2}.
\end{displaymath}
The inequality is sharp: there is equality if (but not only if) $\Gamma$ is an \emph{equilateral flower graph}.
\end{theorem}

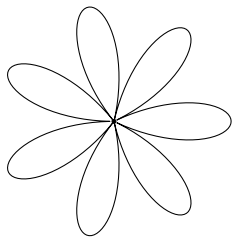
\begin{figure}[ht]
\centering
\begin{tikzpicture}[scale=0.5]
\begin{polaraxis}[grid=none, axis lines=none]
\addplot[mark=none,domain=0:360,samples=300,thick] { abs(cos(7*x/2))};
\draw[fill] (0,0) circle (1.5pt);
\end{polaraxis}
\end{tikzpicture}
\caption{A flower graph with $7$ edges.}
\end{figure}
(A flower graph is a graph with only one vertex, with all edges starting and ending at that vertex. It is equilateral if all edge lengths are equal; in this case they would each have length $L/m$.)

For the proof we need Lemma~\ref{lem:cut} and a second surgery principle:

\begin{definition}
We say that a closed subset $\Gamma' \subset \Gamma$ of a given graph $\Gamma$ is a \emph{pendant subgraph} of $\Gamma$ if it is attached to the rest of $\Gamma$ at a single point, that is, the set
\begin{displaymath}
	\overline{\Gamma'} \cap \overline{\Gamma \setminus \Gamma'}
\end{displaymath}
is a singleton.
\end{definition}

The single point of attachment may without loss of generality be assumed to be a vertex of $\Gamma$, by Remark~\ref{rem:other-vc}.

\begin{lemma}
\label{lem:pendant}
Suppose $\widetilde\Gamma$ is formed from $\Gamma$ by deleting a pendant subgraph $\Gamma' \subset \Gamma$ (equivalently, $\Gamma$ is formed from $\widetilde\Gamma$ by gluing a new graph $\Gamma'$ to $\widetilde\Gamma$ at a single vertex). Suppose standard conditions are imposed at all vertices of all graphs in question. Then
\begin{displaymath}
	\mu_k (\widetilde\Gamma) \geq \mu_k (\Gamma)
\end{displaymath}
for all $k \geq 1$.
\end{lemma}

\begin{proof}
Given $\Gamma$ and its subgraph $\Gamma'$ attached to $\widetilde\Gamma = \Gamma\setminus\Gamma'$ at the vertex $v$, form a new, disconnected graph
\begin{displaymath}
	\Gamma' \sqcup \widetilde\Gamma
\end{displaymath}
via a suitable cut through $v$. This is a cut of rank $1$, thus, by Lemma~\ref{lem:cut},
\begin{displaymath}
	\mu_k (\Gamma' \sqcup \widetilde\Gamma) \geq \mu_{k-1} (\Gamma)
\end{displaymath}
for all $k \geq 2$. Now the set of eigenvalues of $\Gamma' \sqcup \widetilde\Gamma$ is just the union of the set of eigenvalues of $\Gamma'$ and the set of eigenvalues of $\widetilde\Gamma$ (where eigenvalues are always repeated according to their multiplicities). Since $\mu_1 (\Gamma) = \mu_1 (\widetilde\Gamma) = 0$ has multiplicity $2$, it follows that $\mu_2 (\widetilde\Gamma)$ can, at best, correspond to $\mu_3 (\Gamma' \sqcup \widetilde\Gamma)$, $\mu_3 (\widetilde\Gamma)$ at best to $\mu_4 (\Gamma' \sqcup \widetilde\Gamma)$, and so on. Thus, in general,
\begin{displaymath}
	\mu_k (\widetilde\Gamma) \geq \mu_{k+1} (\Gamma' \sqcup \widetilde\Gamma) \geq \mu_{k} (\Gamma)
\end{displaymath}
for all $k \geq 1$.
\end{proof}

This lemma can also be proved using a test function argument, see \cite[Proposition~3.1]{rohleder}. Its sharpest form can be found in \cite[Theorem~3.10]{bkkm2}. It is intuitively clear that attaching a pendant subgraph (basically, an extra dead end) should slow the rate of diffusion. The picture is far more complicated if the subgraph being attached or deleted is not a \emph{pendant}, but rather attached at two or more points, as the following simple example shows.

\begin{example}
Let $\Gamma$ be a loop of length $L>0$, let $\widetilde\Gamma$ be an interval subgraph of $\Gamma$ of length $\ell$, and let $\Gamma'$ be its complement, an interval of length $L-\ell$; then $\Gamma$ is formed by gluing $\widetilde\Gamma$ and $\Gamma'$ twice, at their respective endpoints. Now
\begin{displaymath}
	\mu_2 (\Gamma) = \frac{4\pi^2}{L^2}, \qquad \mu_2 (\widetilde\Gamma) = \frac{\pi^2}{\ell^2},
\end{displaymath}
so that $\mu_2 (\widetilde\Gamma) \geq \mu_2 (\Gamma)$ iff $\ell \leq \frac{L}{2}$.
\end{example}

\begin{proof}[Proof of Theorem~\ref{thm:flower}]
We may certainly assume that $m \geq 2$, since otherwise $\Gamma$ is just an interval. Given such a $\Gamma$, glue all its vertices together. This forms a \emph{flower graph} $F$ with the same number $m$ of edges and same total length $L$ as $\Gamma$, and
\begin{displaymath}
	\mu_2 (\Gamma) \leq \mu_2 (F)
\end{displaymath}
by Lemma~\ref{lem:cut}. Choose the \emph{longest two} edges of $F$, then their total length is at least $2L/m$. Form a new graph $\widetilde F$ by deleting all remaining $m-2$ edges from $F$ (``pluck all but the two largest petals''). Since each edge of $F$ is a pendant subgraph of $F$ attached to the rest of $F$ at a single vertex, by Lemma~\ref{lem:pendant}
\begin{displaymath}
	\mu_2 (F) \leq \mu_2 (\widetilde F).
\end{displaymath} 
Now $\widetilde F$ is a \emph{figure-8 graph}. A direct calculation (left as an exercise, using the result of Exercise~3.2) shows that a figure-8 has the same $\mu_2$ as a circle of the same length:
\begin{displaymath}
	\mu_2 (\widetilde F) = \frac{4\pi^2}{|\widetilde F|^2} \leq \frac{4\pi^2 m^2}{(2L)^2} = \frac{\pi^2m^2}{L^2}.
\end{displaymath}
It is left as another exercise to check that the second eigenvalue of an equilateral $m$-flower graph is in fact $\frac{\pi^2m^2}{L^2}$.
\end{proof}

\subsection{Advanced surgery}
\label{sec:adv-surgery}

To illustrate the full power of the surgery methods introduced in the previous section (and more advanced surgery techniques), we briefly sketch two examples of applications of surgery methods for obtaining bounds on $\mu_2 (\Gamma)$ on the more sophisticated end of the spectrum (no pun intended), following \cite{bkkm2}. For another such application see \cite{kennedy}.

\textbf{Third proof of Theorem~\ref{thm:nicaise}: ``unfolding pendant edges''} (Berkolaiko--K.--Kurasov--Mugnolo)

\begin{definition}
Let $e_1,\ldots,e_k$ be pendant edges of $\Gamma$ (i.e.\ each should have one vertex of degree one) all attached to the same vertex $v$. We say that $\widetilde\Gamma$ is obtained from $\Gamma$ by \emph{unfolding the pendant edges} $e_1,\ldots,e_k$ if these $k$ edges are deleted and replaced by a single pendant edge of length $|e_1|+\ldots+|e_k|$.
\begin{figure}[H]
\centering
\begin{tikzpicture}[scale=0.95]
\draw[thick,dashed] (-5,1) -- (-3.5,0);
\draw[thick,dashed] (-5.5,0) -- (-3.5,0);
\draw[thick,dashed] (-5,-1) -- (-3.5,0);
\draw[thick] (-4.5,0.6667) -- (-2.5,-0.6667);
\draw[thick] (-4.75,0) -- (-2.25,0);
\draw[thick] (-4.5,-0.6667) -- (-2.5,0.6667);
\draw[fill] (-3.5,0) circle (1.5pt);
\node at (-3.5,0) [anchor=north] {$v$};
\draw[fill] (-2.5,-0.6667) circle (1.5pt);
\draw[fill] (-2.25,0) circle (1.5pt);
\draw[fill] (-2.5,0.6667) circle (1.5pt);
\draw[-{Stealth[scale=0.5,angle'=60]},line width=2.5pt] (-0.9,0) -- (0.1,0);
\draw[thick,dashed] (1,0) -- (3,0);
\draw[thick] (1.75,0) -- (3,0);
\draw[thick,dashed] (1.5,1) -- (3,0);
\draw[thick] (2,0.6667) -- (3,0);
\draw[thick,dashed] (1.5,-1) -- (3,0);
\draw[thick] (2,-0.6667) -- (3,0);
\draw[fill] (3,0) circle (1.5pt);
\node at (3,0) [anchor=north] {$v$};
\draw[thick] (3,0) -- (5.75,0);
\draw[fill] (5.75,0) circle (1.5pt);
\end{tikzpicture}
\caption{Unfolding three pendant edges at $v$.}
\end{figure}
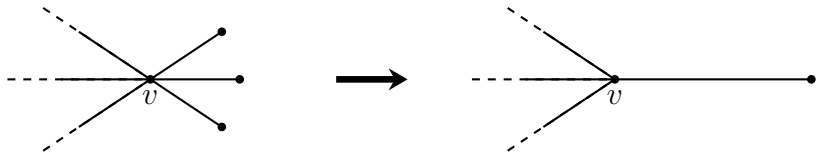
\end{definition}

This is an example of \emph{transplantation}, removing length (``mass'') from one part of the graph (which may be deleting or shortening edges) and attaching it elsewhere (inserting new edges or lengthening additional ones), usually in such a way as to preserve the total length. It is explored in detail in \cite[Section~3.3]{bkkm2}.

\begin{lemma}
\label{lem:unfold-pendant}
Suppose $\widetilde\Gamma$ is formed from $\Gamma$ by unfolding pendant edges. Then
\begin{displaymath}
	\mu_2 (\widetilde\Gamma) \leq \mu_2 (\Gamma).
\end{displaymath}
\end{lemma}

Keeping with our running intuition about the speed of diffusion, this kind of principle (and related principles involving ``unfolding parallel edges'') says that ``diffusion (in our usual $L^2$-sense) is slowed when connections in parallel are replaced by connections in sequence''. The elementary but somewhat delicate proof is given in \cite[Proof of Theorem~3.18(4)]{bkkm2}.

\begin{proof}[Third proof of Theorem~\ref{thm:nicaise}]
Fix a graph $\Gamma$. By Lemma~\ref{lem:cut} we may assume without loss of generality that $\Gamma$ is actually a tree, since cutting through any cycles can only lower $\mu_2$. Denote by $n=n(\Gamma)$ the number of \emph{leaves} (degree one vertices) of $\Gamma$; $\Gamma$ is a path graph if and only if $n=2$. If $n \geq 3$, then $\Gamma$ must have at least one pair of neighbouring pendant edges. Create a new graph $\Gamma_1$ by unfolding these; then we have reduced the number of leaves by one, $n(\Gamma_1) = n-1$, and $\mu_2 (\Gamma_1) \leq \mu_2 (\Gamma)$ by Lemma~\ref{lem:unfold-pendant}. Repeating this argument inductively will produce a path graph, with lower $\mu_2$, after $n-2$ steps.
\end{proof}

The characterisation of equality in Theorem~\ref{thm:nicaise} can be obtained by using a refined version of Lemma~\ref{lem:unfold-pendant}. This proof, while not being generalisable to the higher eigenvalues or more highly connected graphs, does give a family of comparisons: $\mu_2$ decreases successively as the graph is progressively transformed from its original state into a more path-like object.\bigskip

\textbf{Interpolation between the theorems of Nicaise and Band--L\'evy; fourth proof of Theorem~\ref{thm:nicaise}} (Berkolaiko--K.--Kurasov--Mugnolo)

We recall that Nicaise' theorem (Theorem~\ref{thm:nicaise}) states that
\begin{displaymath}
	\mu_2 (\Gamma) \geq \frac{\pi^2}{L^2}
\end{displaymath}
for all compact graphs $\Gamma$ of total length $L$, equality being achieved (only) by path graphs (equivalently, intervals); while the Theorem of Band--L\'evy (\cite[Theorem~2.1(2)]{bandlevy}, cf.\ Remark~\ref{rem:eulerian-cycle}) states that
\begin{displaymath}
	\mu_2 (\Gamma) \geq \frac{4\pi^2}{L^2}
\end{displaymath}
for all doubly connected graphs $\Gamma$ of total length $L$, equality being achieved by loops, equivalently, flat one-dimensional tori.\footnote{There is actually a family of related minimisers owing to a basic surgery principle, but not relevant for our purposes here: we may glue the loop together at certain points without affecting the eigenvalue, if there is always an eigenfunction which takes on the same value at the points being glued. This is the refinement of Lemma~\ref{lem:cut} mentioned in the proof of Theorem~\ref{thm:flower}.} We can interpolate between the two inequalities based on the size of the \emph{doubly connected part} of $\Gamma$:

\begin{definition}
The \emph{doubly connected part} $\mathcal{D}_\Gamma$ of a graph $\Gamma$ is the unique closed subgraph of $\Gamma$ such that $x \in \mathcal{D}_\Gamma$ if and only if there exists a non-self-intersecting path in $\Gamma$ starting and ending at $x$.
\end{definition}

Equivalently, $\mathcal{D}_\Gamma$ is the largest subgraph of $\Gamma$ (not necessarily connected) such that every connected component is itself doubly connected; $\mathcal{D}_\Gamma$ may be obtained from $\Gamma$ by deleting every point $x \in \Gamma$ whose removal would disconnect $\Gamma$ (``bridges'').

\begin{figure}[ht]
\begin{center}
\begin{tikzpicture}[scale=0.8]
\coordinate (a) at (0,0);
\coordinate (b) at (2,0);
\coordinate (c) at (5,0);
\coordinate (d) at (8,0);
\draw[fill] (2,0) circle (1.5pt);
\draw[fill] (5,0) circle (1.5pt);
\draw[thick,bend left=90] (a) edge (b);
\draw[thick,bend right=90] (a) edge (b);
\draw[thick,bend left=90] (c) edge (d);
\draw[thick,bend right=90] (c) edge (d);
\draw[thick] (b) -- (c);
\draw[thick,dashed] (9,-1) -- (9,1);
\coordinate (e) at (10,0);
\coordinate (f) at (12,0);
\coordinate (g) at (15,0);
\coordinate (h) at (18,0);
\draw[fill] (12,0) circle (1.5pt);
\draw[fill] (15,0) circle (1.5pt);
\draw[thick,bend left=90] (e) edge (f);
\draw[thick,bend right=90] (e) edge (f);
\draw[thick,bend left=90] (g) edge (h);
\draw[thick,bend right=90] (g) edge (h);
\end{tikzpicture}
\caption{A dumbbell graph (left); its doubly connected part consists of its two loops (right).}
\label{fig:dumbbell}
\end{center}
\end{figure}
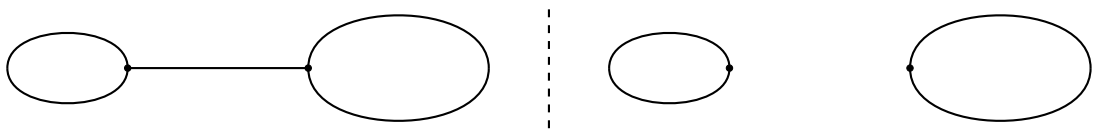

\begin{theorem} {\rm \cite[Theorem~6.3]{bkkm2}}
\label{thm:doubly-connected-part}
Suppose $\Gamma$ has doubly connected part of total length $V \in [0,L]$. Let $\mathfrak{D}$ be the dumbbell graph (cf.\ Figure~\ref{fig:dumbbell}) of total length $L$, whose two loops have equal length $V/2$ each. Then
\begin{displaymath}
	\mu_2 (\Gamma) \geq \mu_2 (\mathfrak{D}).
\end{displaymath}
\end{theorem}

While in general there is no closed analytic formula for $\mu_2 (\mathfrak{D})$ as a function of $L$ and $V$, one can show easily that, for fixed $L$, it is a monotonically increasing function of $V$ from $\frac{\pi^2}{L^2}$ at $V=0$ to $\frac{4\pi^2}{L^2}$ at $V=L$ (where $\mathfrak{D}$ becomes a \emph{figure-8} graph); in particular, it interpolates smoothly between, and contains as special cases, the bounds of Nicaise and Band--L\'evy in function of $V=|\mathcal{D}_\Gamma|$.\footnote{In \cite{bkkm2} these claims about $\mu_2 (\mathfrak{D})$ are proved using surgery techniques -- what else? -- however the secular equation in this case should be just about simple enough to study directly.}

The proof of Theorem~\ref{thm:doubly-connected-part} (which, in particular, yields a fourth proof of Theorem~\ref{thm:nicaise}) combines all the techniques we have seen thus far, and variants:

\textbf{Step 1:} Use the previously mentioned case of equality in Lemma~\ref{lem:cut} to glue $\Gamma$ at all ``critical levels'', i.e.\ one glues together all points at which the eigenfunction has a critical value or a vertex value; this creates a ``minimally maximally connected'' graph with the same $\mu_2$, known as a \emph{pumpkin chain}, and on which the eigenfunction is monotonically increasing along the chain. (See Figure~\ref{fig:pumpkin-chainification}.)

This process preserves the size of the doubly connected part; any bridges of $\Gamma$ are preserved as bridges in the pumpkin chain. 

\textbf{Step 2:} The edges within each of the constituent pumpkins in the \emph{pumpkin chain} from Step 1 may not have the same length. Apply a \emph{local symmetrisation argument} (a local version of the symmetrisation method of Section~\ref{sec:nicaise}) to make each of the pumpkins equilateral (although edges from different pumpkins are allowed to have different lengths). In Figure~\ref{fig:pumpkin-chainification} this step is not necessary due to the symmetry of the original graph and its eigenfunction.

\textbf{Step 3:} Apply a transplantation argument to the symmetrised pumpkin chain to ``push'' the mass outwards to form a dumbbell.\footnote{This step is actually quite complicated and is performed individually on each of the two nodal domains of the pumpkin chain. While it is a transplantation, in practice it requires combining a transplantation argument with a cutting procedure and a \emph{Hadamard formula} for the derivative of $\mu_2$ as a function of the edge lengths.}

\textbf{Step 4:} This dumbbell may not be symmetric. However, ``balancing'' the loops to create the symmetric dumbbell of the theorem lowers $\mu_2$.\footnote{This, again, requires the Hadamard formula mentioned in the previous footnote.}

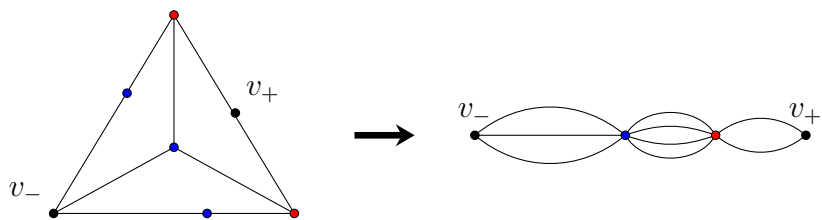
\begin{figure}[ht]
\begin{center}
\begin{tikzpicture}[scale=0.8]

\coordinate (a) at (1,-1.3);
\coordinate (b) at (5,-1.3);
\coordinate (c) at (3,-.2);
\coordinate (d) at (3,2);
\coordinate (e) at (4.02,.37);
\coordinate (e1) at (3.55,-1.3);
\coordinate (e2) at (2.22,.7);
\draw (a) -- (b);
\draw (b) -- (c);
\draw (c) -- (d);
\draw (a) -- (d);
\draw (a) -- (c);
\draw (b) -- (d);

\draw[fill] (a) circle (2pt);
\draw[fill=red] (b) circle (2pt);
\draw[fill=blue] (c) circle (2pt);
\draw[fill=red] (d) circle (2pt);
\draw[fill] (e) circle (2pt);
\draw[fill=blue] (e1) circle (2pt);
\draw[fill=blue] (e2) circle (2pt);

\node at (a) [anchor=south east ] {$v_-$};
\node at (e) [anchor=south west] {$v_+$};

\draw[-{Stealth[scale=0.5,angle'=60]},line width=2.5pt] (6,0) -- (7,0);

\coordinate (g) at (8,0);
\coordinate (h) at (10.5,0);
\coordinate (i) at (12,0);
\coordinate (j) at (13.5,0);

\draw[bend left=60] (h) edge (i);
\draw[bend left=20] (h) edge (i);
\draw[bend left=-60] (h) edge (i);
\draw[bend left=-20] (h) edge (i);
\draw[bend left=40] (g) edge (h);
\draw (g) -- (h);
\draw[bend right=40] (g) edge (h);
\draw[bend right=40] (i) edge (j);
\draw[bend right=-40] (i) edge (j);

\node at (g) [anchor=south] {$v_-$};
\node at (j) [anchor=south] {$v_+$};

\draw[fill] (g) circle (2pt);
\draw[fill=blue] (h) circle (2pt);
\draw[fill=red] (i) circle (2pt);
\draw[fill] (j) circle (2pt);
\end{tikzpicture}
\end{center}
\caption{Turning a graph (here an equilateral complete graph on four edges) into a \emph{pumpkin chain}, based on the behaviour of an arbitrary but fixed second eigenfunction. The vertices $v_-$ and $v_+$ represent, respectively, where the eigenfunction reaches its minimum and maximum; its values at the other vertices are considered ``critical'', and thus all points at which it attains those values (labelled in blue and red, respectively) are glued together. The resulting pumpkin chain has the same eigenvalue and eigenfunction (up to a canonical mapping between the graphs), such that the eigenfunction is now monotonically increasing along each path from $v_-$ to $v_+$. (Each of the sets of parallel edges betweek neighbouring vertices is a \emph{pumpkin}.) Figure adapted from \cite[Figure~5.1]{bkkm2}.}
\label{fig:pumpkin-chainification}
\end{figure}

\section*{Appendix: Exercises}
\label{chapter:exercises}
\addcontentsline{toc}{chapter}{Exercises}

\begin{itemize}
\item[\textbf{1.1}]\label{ex:1.1} Use the spectral theorem to prove Courant--Fischer (Theorem~\ref{thm:courant-fischer}), and show that the minimum in \eqref{eq:l1} is attained if and only if $u$ is an eigenvector corresponding to $\lambda_1$. Can you \emph{characterise} equality in \eqref{eq:lk-minmax} and \eqref{eq:lk-maxmin}?\medskip
\item[\textbf{1.2}]\label{ex:1.2} Prove that if $u \in C^1(\overline{\Omega})$, then its classical partial derivative $\frac{\partial u}{\partial x_j} \in C(\overline{\Omega})$ is also a weak partial derivative of $u$. Conversely, prove that if $g_j$ is a weak partial derivative of $u$ and $g_j \in C(\overline{\Omega})$, then in fact $g_j$ is a classical partial derivative of $u$.\medskip
\item[\textbf{1.3}]\label{ex:1.3} Prove Lemma~\ref{lem:aq-omega}.\medskip
\item[\textbf{1.4}]\label{ex:1.4} \textbf{Schr{\"o}dinger operators.} Given a bounded Lipschitz domain $\Omega \subset \R^d$, define a Hermitian form $a_q : H^1_0 (\Omega) \times H^1_0 (\Omega) \to \C$ for $q \in L^\infty (\Omega,\R)$  by
\begin{displaymath}
	a_q (u,v) := \int_\Omega \nabla u \cdot \overline{\nabla v} + qu\overline{v}\,\textrm{d}x.
\end{displaymath}
\begin{itemize}
\item[\textbf{(a)}] Show that the form $a_q$ is bounded and $L^2$-elliptic on $H^1_0(\Omega)$, and find the operator $A_q$ on $L^2(\Omega)$ which is associated with $a_q$. Deduce that $A_q$ satisfies the conclusions of the spectral theorem, Theorem~\ref{thm:spectral}.\\
\emph{Note:} In particular, this proves Theorem~\ref{thm:laplacians-omega}(a), which corresponds to $q=0$.
\item[\textbf{(b)}] Use the variational characterisation of the eigenvalues to show that, given $q_1,q_2 \in L^\infty (\Omega,\R)$, if $q_1 \leq q_2$ almost everywhere in $\Omega$, then
\begin{displaymath}
	\lambda_k (A_{q_1}) \leq \lambda_k (A_{q_2}) \qquad \text{for all } k \geq 1.
\end{displaymath}
\end{itemize}
\item[\textbf{1.5}]\label{ex:1.5} \textbf{The Robin Laplacian.} Given a bounded Lipschitz domain $\Omega \subset \R^d$, for $\alpha \in \R$ define a form $a_\alpha : H^1(\Omega) \times H^1(\Omega) \to \C$ by
\begin{displaymath}
	a_\alpha (u,v) := \int_\Omega \nabla u \cdot \overline{\nabla v}\,\textrm{d}x + \int_{\partial\Omega} \alpha u \overline{v}\,\textrm{d}s,
\end{displaymath}
where the functions in the boundary integral are to be taken as the traces of $u$ and $\overline{v}$. Clearly $a_\alpha$ is Hermitian, and $a_\alpha$ coincides with the form $a$ from \eqref{eq:aq-omega} when $\alpha=0$.
\begin{itemize}
\item[\textbf{(a)}] Use the trace inequality \eqref{eq:trace-omega} to prove that $a_\alpha$ is bounded and $L^2$-elliptic on $H^1(\Omega)$.
\item[\textbf{(b)}] Show that the operator $A_\alpha$ on $L^2(\Omega)$ associated with $a_\alpha$ is given by
\begin{displaymath}
\begin{aligned}
	D(A_\alpha) &= \left\{ u \in H^1(\Omega): \Delta u \in L^2(\Omega),\, \frac{\partial u}{\partial \nu} \in L^2(\partial\Omega) \text{ and }
	=-\alpha u \text{ in } L^2(\partial\Omega)\right\},\\
	A_\alpha u &= -\Delta u,
\end{aligned}
\end{displaymath}
and deduce that $A_\alpha$ satisfies the conclusions of the spectral theorem, Theorem~\ref{thm:spectral}.
\end{itemize}
\item[\textbf{1.6}]\label{ex:1.6} The first eigenvalue of the Robin Laplacian $A_\alpha$ from the previous exercise may be characterised variationally as
\begin{displaymath}
	\lambda_1(\alpha) := \lambda_1(A_\alpha) = \min_{0\neq u \in H^1(\Omega)} \frac{\int_\Omega |\nabla u|^2\,\textrm{d}x
	+ \int_{\partial\Omega} \alpha |u|^2\,\textrm{d}x}{\int_\Omega |u|^2\,\textrm{d}x}.
\end{displaymath}
\begin{itemize}
\item[\textbf{(a)}] Prove that (on a fixed domain $\Omega$) the function $\alpha \mapsto \lambda_1(\alpha)$ is a convex monotonically increasing and continuous function of $\alpha \in \R$, and that $\lambda_1 (\alpha) \leq \lambda_1 (-\Delta^D)$, the first eigenvalue of the Dirichlet Laplacian, for all $\alpha \in \R$.
\item[\textbf{(b)}] Now suppose $\alpha>0$. 

It follows immediately from \textbf{(a)} that $\lambda_1(\alpha) \searrow \lambda_1 (-\Delta^N) = 0$, the first eigenvalue of the Neumann Laplacian, as $\alpha \to 0$. Show that $\lambda_1(\alpha) > 0$ for all $\alpha > 0$, and deduce that $\sqrt{a_\alpha (u,u)}$ defines an equivalent norm on $H^1(\Omega)$.
\item[\textbf{(c)}] Show that $\lambda_1 (\alpha) \nearrow \lambda_1 (-\Delta^D)$ as $\alpha \to \infty$.\\
\emph{Hint:} Given $\alpha_n \to \infty$, let $\psi_n$ be an eigenfunction for $\lambda_1(\alpha_n)$ with $L^2$-norm~$1$. Deduce from \textbf{(a)} and \textbf{(b)} that $(\psi_n)$ forms a bounded, and thus weakly convergent, sequence in $H^1(\Omega)$.
\end{itemize}
\medskip
\item[\textbf{2.1}]\label{ex:2.3} \textbf{$W^{1,p}_0$-spaces.} Given a (connected, compact) graph $\Gamma = (V,E)$, choose an arbitrary subset $V_D \subseteq V$ of the vertices of $\Gamma$ and define, for fixed $1 \leq p \leq \infty$,
\begin{displaymath}
	W^{1,p}_0 (\Gamma) := W^{1,p}_0 (\Gamma; V_D) := \{f \in W^{1,p}(\Gamma): f(v) = 0 \text{ for all } v \in V_D \}.
\end{displaymath}
If $p=2$, then we usually write $H^1_0 (\Gamma) = H^1_0 (\Gamma; V_D)$ in place of $W^{1,2}_0 (\Gamma) = W^{1,2}_0 (\Gamma; V_D)$.
\begin{itemize}
\item[\textbf{(a)}] Prove that $W^{1,p}_0 (\Gamma; V_D)$ is a closed subspace of $W^{1,p}(\Gamma)$.\\
\emph{Hint:} It follows from Lemma~\ref{lem:embedding-1d} that the embedding $W^{1,p}(\Gamma) \hookrightarrow C(\Gamma)$ is also continuous.
\item[\textbf{(b)}] Consider the $3$-star graph $\Gamma$ of Example~\ref{ex:3-star}. Show that if $V_D = \{v_4\}$, then $W^{1,p}_0 (\Gamma; \{v_4\})$ is isometrically isomorphic to the direct sum of three copies of $W^{1,p}_0 ([0,1]; \{1\})$. Thus, imposing a Dirichlet condition at a vertex of degree $\geq 2$ ``decouples'' the vertex.
\item[\textbf{(c)}] Still on the example of the $3$-star graph $\Gamma$, show that $W^{1,p}_0 (\Gamma; \{v_1,v_2,v_3\})$ is not isometrically isomorphic to $W^{1,p}_0 (\Gamma; \{v_4\})$. Obviously, the choice of the set $V_D$ matters!
\item[\textbf{(d)}] Returning to the general case, prove that the codimension of $H^{1}_0 (\Gamma)$ in $H^{1} (\Gamma)$ is $\#V_D$. Is the same true if $p\neq 2$?
\end{itemize}\medskip
\item[\textbf{2.2}]\label{ex:2.4} Prove Lemma~\ref{lem:form-graph}.\medskip
\item[\textbf{2.3}]\label{ex:2.5} Prove Proposition~\ref{prop:standard-laplacian} in the special case when $q=0$.\medskip
\item[\textbf{2.4}]\label{ex:2.6} \textbf{The Laplacian with (some) Dirichlet vertex conditions.} Given a (connected, compact) graph $\Gamma = (V,E)$, choose a set $V_D \subseteq V$ and consider the form
\begin{displaymath}
	a_0 (f,g) = \int_\Gamma f'\overline{g'}\,\textrm{d}x
\end{displaymath}
from \eqref{eq:form-graph}, but this time defined on $H^1_0(\Gamma;V_D)$ instead of $H^1(\Gamma)$. It is obvious that Lemma~\ref{lem:form-graph} continues to hold in this space, and thus the spectral theorem holds for the associated operator.
\begin{itemize}
\item[\textbf{(a)}] Prove that the associated operator $A_0^D$ is given by
\begin{displaymath}
\begin{aligned}
	D(A_0^D) &= \big\{ f \in H^1(\Gamma): f'' \in L^2(\Gamma),\,\, f(v)=0\,\, \forall v \in V_D,\\
	& \qquad\quad\text{ and } \sum_{e\sim v} \partial_\nu f|_e(v)=0\,\, \forall v \in V\setminus V_D\big\},\\
	A_0^D f &= -f''
\end{aligned}
\end{displaymath}
in the sense of distributions. Thus, functions in $D(A_0^D)$ satisfy a Dirichlet condition at every vertex in $V_D$ and standard conditions at all others.
\item[\textbf{(b)}] Prove that, for any $V_D \subseteq V$,
\begin{displaymath}
	\lambda_k(A_0) \leq \lambda_k (A_0^D) \leq \lambda_{k+(\#V_D)} (A_0),
\end{displaymath}
where the $\lambda_k(A_0)$ are the eigenvalues of the standard Laplacian.\\
\emph{Hint:} Use the results of \textbf{2.1}\textbf{(a)} and \textbf{(d)} together with the variational (min-max) characterisation of the eigenvalues.
\item[\textbf{(c)}] Prove that if $V_D \neq \emptyset$, then $\lambda_1 (A_0^D) > 0$.
\item[\textbf{(d)}] Repeat this exercise (with appropriate modifications) for the following operator. Given a set of vertices $V_R \subseteq V$ and $\alpha > 0$, the form is taken to be
\begin{displaymath}
	a_\alpha (f,g) = \int_\Gamma f'\overline{g'}\,\textrm{d}x + \sum_{v \in V_R} \alpha f(v) \overline{g(v)}
\end{displaymath}
on $H^1(\Gamma)$.\\
\emph{Note:} The associated operator is the Laplacian with \emph{Robin}, a.k.a.\ $\delta$, vertex conditions on $V_R$ (and standard conditions elsewhere). The constant $\alpha$, which may also be allowed to depend on the vertex, represents a $\delta$-potential being imposed at the vertices $V_R$. This constant may also be taken non-positive, but then the conclusions of \textbf{(b)} and \textbf{(c)} are not valid. The eigenvalues do, however, satisfy the same conclusions as in \textbf{1.6}.
\end{itemize}\medskip
\item[\textbf{2.5}]\label{ex:2.7} \textbf{The discrete Laplacian} on the $3$-star graph. We will study a prototypical discrete graph which is the analogue of the $3$-star graph of Example~\ref{ex:3-star}.
\begin{itemize}
\item[\textbf{(a)}] With the vertex numbering of Example~\ref{ex:3-star} (in particular, where $\deg v_4 = 3$), write down the \emph{adjacency matrix} $A$ of $\Gamma$, that is, $A_{ij} = 1$ if there is an edge between $v_i$ and $v_j$, and $0$ otherwise.
\item[\textbf{(b)}] Let $D$ be the $4\times 4$-diagonal matrix whose $(i,i)$-th entry is the degree of the vertex $v_i$, and let
\begin{displaymath}
	L = D - A.
\end{displaymath}
Determine $L$, and determine its spectrum. $L$ is the \emph{discrete Laplacian} on $\Gamma$. (Observe that $L$ is a symmetric, real-valued matrix.)
\item[\textbf{(c)}] Show that a vector $u \in \C^4$ satisfies $Lu = 0$ if and only if $u$ has the \emph{mean value property}: for every vertex $v_k$,
\begin{displaymath}
	u(v_k) = \frac{1}{\deg v_k} \sum_{v_j \sim v_k} u(v_j).
\end{displaymath}
\end{itemize}
These definitions and properties hold for the discrete Laplacian on \emph{any} finite graph $\Gamma$.\medskip
\item[\textbf{3.1}]\label{ex:3.1} Let $\Gamma_m$ be the (equilateral) \emph{$m$-star graph} consisting of $m$ edges of length $\frac{1}{m}$ each, attached together at a vertex of degree $m$. (That is, formally, $E=\{e_1,\ldots,e_m\}$, $V=\{v_1,\ldots,v_{m+1}\}$, with $e_k \sim v_k v_{m+1}$, where $\deg v_k = 1$ for $k=1,\ldots,m$ and $\deg v_{m+1} = m$.)
\begin{itemize}
\item[\textbf{(a)}] Determine the lowest nontrivial eigenvalue $\mu_2 (\Gamma_{m})$ of the standard Laplacian on $\Gamma_m$, and conclude that $\mu_2 (\Gamma_{m}) \to \infty$ even though $|\Gamma_m|=1$ for all $m \geq 3$.
\item[\textbf{(b)}] Describe all eigenvalues and eigenfunctions of $\Gamma_m$. Take care with the high multiplicities.
\end{itemize}
\item[\textbf{3.2}]\label{ex:3.2} Use the characterisation of equality in \eqref{eq:lambda2-rq-graph} to prove the following special case of equality in Lemma~\ref{lem:cut}: Suppose $\Gamma$ is formed from $\widetilde\Gamma$ by gluing together two vertices $v_1$ and $v_2$, and suppose also that there is an eigenfunction $\psi_2$ for $\mu_2(\widetilde\Gamma)$ such that $\psi_2(v_1) = \psi_2(v_2)$. Then $\mu_2(\Gamma) = \mu_2 (\widetilde\Gamma)$.\medskip
\item[\textbf{3.3}]\label{ex:3.3} \textbf{Weyl asymptotics.} In this exercise we will prove Corollary~\ref{cor:weyl-graph}, that if $\Gamma$ has total length $L$, then $\mu_k (\Gamma) = \frac{\pi^2 k^2}{L^2} + o(k^2)$ as $k \to \infty$.
\begin{itemize}
\item[\textbf{(a)}] Prove that the assertion holds if $\Gamma$ is just an interval of length $L$.
\item[\textbf{(b)}] Use \textbf{(a)} to show that the assertion also holds if $\Gamma$ is a disjoint union of $m$ intervals (of lengths $\ell_1,\ldots,\ell_m > 0$).\\
\emph{Hint:} It may be more practical to work with the \emph{counting function}
\begin{displaymath}
	N(\lambda) := \#\{k \in \N: \mu_k (\Gamma) \leq \lambda \}.
\end{displaymath}
\item[\textbf{(c)}] Use \textbf{(b)} and \eqref{eq:cut} to obtain the Weyl asymptotics for an arbitrary compact metric graph $\Gamma$.
\item[\textbf{(d)}] Prove that the Weyl asymptotics is still valid if any mix of standard, Dirichlet and Robin conditions (with $\alpha>0$ for simplicity) is imposed at the vertices.
\end{itemize}
\item[\textbf{3.4}] \begin{itemize}\item[\textbf{(a)}] Prove \emph{Rohleder's inequality} \cite[Theorem~3.4]{rohleder}: Suppose $\Gamma$ is a \emph{tree graph}, i.e. a graph with no cycles, and define its \emph{diameter} to be the longest path in the graph,
\begin{displaymath}
	D:= \max \{\dist (x,y): x,y \in \Gamma \}.
\end{displaymath}
Then, for any $k \geq 1$,
\begin{displaymath}
	\mu_k (\Gamma) \leq \frac{\pi^2k^2}{D^2} = \mu_k (\text{interval of length $D$}).
\end{displaymath}
\item[\textbf{(b)}] (Hard.) Does the same inequality hold if $\Gamma$ is not a tree graph?
\item[\textbf{(c)}] (Less hard but still hard.) Prove that no corresponding lower bound is possible. That is, find a sequence of tree graphs $\Gamma_n$ such that each has diameter $1$, but $\mu_2 (\Gamma_n) \to 0$ as $n \to \infty$.
\end{itemize}
\end{itemize}

{\small\bibliography{bibliography}}

\EditInfo{October 9, 2023}{January 19, 2024}{Ana Cristina Moreira Freitas, Diogo Oliveira e Silva, Ivan Kaygorodov, Carlos Florentino}

\end{document}